\documentclass{amsart}
\usepackage{amssymb}
\usepackage{graphicx}

\def\alp{\alpha}

\def\dd{\displaystyle}

\def\gam{\gamma}
\def\Gam{\Gamma}

\def\mtline#1{\hbox to#1{\hrulefill}}

\def\noi{\noindent}

\def\ome{\omega}

\def\what{\widehat}
\def\wtit{\widetilde}

\def\cA{{\mathcal A}}

\def\cH{{\mathcal H}}

\def\cK{{\mathcal K}}

\def\cR{{\mathcal R}}
\def\cS{{\mathcal S}}

\def\cU{{\mathcal U}}

\def\maA{{\mathcal A}}

\def\maH{{\mathcal H}}

\def\maK{{\mathcal K}}

\def\maR{{\mathcal R}}
\def\maS{{\mathcal S}}

\def\maU{{\mathcal U}}

\def\C{{\mathbb C}}
\def\D{{\mathbb D}}

\def\N{{\mathbb N}}

\def\R{{\mathbb R}}
\def\S{{\mathbb S}}

\def\Z{{\mathbb Z}}

\def\fm{\mathfrak m}

\usepackage{secteqn}
\usepackage{mathabx}

\usepackage{xcolor}
\definecolor{purple}{cmyk}{.33,1,0,.4}
\definecolor{m}{rgb}{1,0.1,1}
\definecolor{green}{cmyk}{1,0,1,0}
\definecolor{test}{rgb}{1,1,1}
\definecolor{cmyk}{cmyk}{0,1,1,0}

\newtheorem{Equation}{}[section]

\newtheorem{corollary}[Equation]{Corollary}
\newtheorem{definition}[Equation]{Definition}

\newtheorem{lemma}[Equation]{Lemma}
\newtheorem{proposition}[Equation]{Proposition}

\newtheorem{remark}[Equation]{Remark}

\newtheorem{theorem}[Equation]{Theorem}

\newtheorem{assu}[Equation]{Assumption}

\newtheorem{coro}[Equation]{Corollary}
\newtheorem{defi}[Equation]{Definition}

\newtheorem{lemm}[Equation]{Lemma}
\newtheorem{prop}[Equation]{Proposition}

\newtheorem{theo}[Equation]{Theorem}

\def\AS{\operatorname{AS}}
\def\ch{\operatorname{ch}}

\def\dim{\operatorname{dim}}
\def\Dim{\operatorname{Dim}}

\def\div{\operatorname{div}}

\def\Id{\operatorname{I}}

\def\Im{\operatorname{Im}}

\def\Ind{\operatorname{Ind}}

\def\Ker{\operatorname{Ker}}

\def\Supp{\operatorname{Supp}}

\def\tr{\operatorname{tr}}
\def\htr{\operatorname{\mathfrak{tr}}}
\def\vol{\operatorname{vol}}

\def\cA{{\mathcal A}}

\def\maA{{\mathcal A}}

\def\alp{\alpha}

\def\gam{\gamma}
\def\Gam{\Gamma}

\def\ep{\epsilon}

\def\Lam{\Lambda}

\def\ome{\omega}

\def\dd{\displaystyle}
\def\pa{\partial}

\def\lan{\langle}
\def\ran{\rangle}
\def\ssm{\smallsetminus}

\def\what{\widehat}
\def\wtit{\widetilde}

\begin{document}



\title[ Dirac operators on foliations with invariant transverse  measures\ \today]
{ Dirac operators on foliations \\ with invariant  transverse measures  \\  \today}


\author{Moulay Tahar Benameur}
\address{Institut Montpellierain Alexander Grothendieck, UMR 5149 du CNRS, Universit\'e de Montpellier}
\email{moulay.benameur@umontpellier.fr}

\author[James L. Heitsch \today]{James L.  Heitsch}
\address{Mathematics, Statistics, and Computer Science, University of Illinois at Chicago} 
\email{heitsch@uic.edu}

\thanks{\hspace{-0.5cm}
}

\begin{abstract} We extend the groundbreaking results of Gromov and Lawson, \cite{GL3}, to Dirac operators defined along the leaves of foliations of non-compact complete Riemannian manifolds which admit invariant transverse measures.  We prove a relative measured index theorem for pairs of such manifolds, foliations and operators, which are identified off compact subsets of the manifolds.  We assume that the spectral projections of the leafwise operators for some interval $[0,\ep]$, $\ep > 0$, have finite dimensional images when paired with the invariant transverse measures.  As a prime example, we show that if the zeroth order operators in the associated Bochner Identities are uniformly positive off compact subsets of the manifolds, then they satisfies the hypotheses of our relative measured index theorem.  Using these results, we show that for a large collection of spin foliations, the space of positive scalar curvature metrics on each foliation has infinitely many path connected components. 

\end{abstract} 

\maketitle

\tableofcontents

\section{Introduction}

In this paper we extend some of the groundbreaking results of Gromov and Lawson to Dirac operators defined along the leaves of a foliation $F$ of a non-compact complete Riemannian manifold $M$.   In particular, we extend their highly important relative index theorem, Theorem 4.18 of \cite{GL3}, to this situation, which has been an open problem since the 1980s. That theorem has played a fundamental role in the development and understanding of the existence and non-existence of metrics with positive scalar curvature, as well as the structure of the spaces of such metrics.  It is essential for the extension of results for compact manifolds to non-compact manifolds.  Our work is in the spirit of the transition from the Atiyah-Singer index theorem, \cite{ASIII}, to Connes' measured index theorem for foliations, \cite{ConnesIntegration}.  In order to overcome the problems of dealing with non-compact manifolds, we assume that our objects have bounded geometry.  This, and a good deal of hard analysis, allows us  to prove our first main result, the relative measured index theorem for foliations.   Our second main result is that if the zeroth order operator in the associated Bochner Identity is uniformly positive off a compact subset of $M$, then $(M,F)$ satisfies the hypothesis of our relative index theorem.  We then use these results to show that for a  large collection of spin foliations, the space of positive scalar curvature metrics on each foliation has infinitely many path connected components.

\medskip

The metric on $M$ induces one on the leaves of $F$, and we assume that $M$ and all the leaves of $F$ have bounded geometry.  We also assume that $(M,F)$ admits an invariant transverse Borel measure $\Lam$.  We denote by $d\mu$ the global measure on $M$ determined by $\Lam$ and the leafwise metric.  When the manifold is compact, our results reduce to Connes' index theorem for measured foliations \cite{ConnesIntegration}.

Any Clifford bundle $E$ over the Clifford algebra of the co-tangent bundle to $F$, along with a Hermitian connection $\nabla$ compatible with Clifford multiplication, determines a leafwise Dirac operator, denoted
$$
D_L^E:C^{\infty}_{c}(E) \to C^{\infty}_{c}(E).
$$ 
There is a canonical zeroth order operator $\maR^E_F$ defined on $C^{\infty}_{c}(E)$, so that the operators $D_L^E$, $\cR^E_F$, $\nabla$ and its leafwise formal adjoint $\nabla^*$, are related by the general Bochner Identity, \cite{LM},
$$
 (D_L^E)^2 \,\, = \,\, \nabla^* \nabla \,\, + \,\,  \cR^E_F.
$$ 
 
\medskip

Our first main result is the foliation relative measured index theorem.  In particular, we assume that we have two foliated manifolds $(M, F)$ and $(M', F')$ as above, with invariant  transverse measures $\Lam$ and $\Lam'$, and Clifford bundles $E$ and $E'$.   So there are leafwise Dirac operators $D_L^E$ and $D_L^{E'}$.  We further assume that there are compact subspaces $\cK = M\smallsetminus V$ and $\cK' = M'\smallsetminus V'$ so that the situations on $V$ and $V'$ are identical.   Using parametrices, we can then define a relative measured index for the pair $(D_L^E, D_L^{E'})$, denoted $\Ind_{\Lam,\Lam'}(D_L^E, D_L^{E'})$.   

The first half of the foliation relative measured index theorem is the following.

\medskip

\noi
{\bf Theorem \ref{Firstmain} }  {\em $\Ind_{\Lam, \Lam'}(D_L^E, D_L^{E'})$ is finite, and the following formula holds, \\
$$
\Ind_{\Lam, \Lam'}(D_L^E, D_L^{E'}) \,\, = \,\,
\int_{\cK} \AS (D_L ^E)_L  \, d\Lam \,\, - \,\,   \int_{\cK'} (\AS D_L ^{E'})_L  \, d\Lam',
$$
where  $\AS (D_L ^E)_L $ is the characteristic differential form on the leaves of $F$ associated to $D_L^E$ by the local Atiyah-Singer Index Theorem, and similarly for $\AS (D_L ^{E'})_L$.}

\medskip  

In general, it is not possible to express the relative measured index $\Ind_{\Lambda, \Lambda'} (D_L^E, D_L^{E'})$ in terms of the leafwise  projections $P_0$ and $P'_0$ onto the leafwise kernels of $(D_L^E)^2$ and $(D_L^{E'})^2$, as occurs in the classical cases.  This is because, in general on non-compact manifolds of bounded geometry,  the super-traces of the leafwise Schwartz kernels of $e^{-t (D_L^E)^2}$ and $P_0$ satisfy 
$$
\lim_{t \to \infty} \tr_s( k_{e^{-t (D_L^E)^2}}(x,x))  \,\, = \,\   \tr_s( k_{P_0}(x,x)),
$$
{\em only  pointwise}, and similarly for $e^{-t (D_L^{E'})^2}$ and $P_0'$.  We give conditions here where such an expression is possible.  

\medskip

Denote the leafwise spectral projection associated to $(D_L^E)^2 $ for the interval $[0,\ep]$ by $P_{[0,\ep]}$, and its  
leafwise Schwartz kernel by $k_{P_{[0,\ep]}}$.  The $\Lam$ dimension of the  image of $P_{[0,\ep]}$ is the element of $[0,\infty]$, given by
$$
\Dim_{\Lam}(\Im(P_{[0,\ep]}))   \,\,  = \,\,  \int_M  \tr(k_{P_{[0,\ep]}}(x, x))  \,  d\mu,
$$
and similarly for $(D_L^{E'})^2$ and  $\Dim_{\Lam}(\Im(P_{[0,\ep]}'))$.  

\medskip

The second half of the foliation relative measured index theorem is the following.   For this theorem we need Assumption \ref{Assume}, which is essentially  that there is an open submanifold $\wtit{M}$ of $M \ssm \cK$, with compact complement,  so that the projection $\wtit{P}$ to the kernel of $D_L^E$ restricted to $\wtit{M}$ satisfies $\dd \int_{\wtit{M}}  \tr(k_{\wtit{P}}(x, x))  \,  d\mu$ is finite.  

\medskip

\noi
{\bf Theorem \ref{mainSection4}} 
{\em Suppose that there is $\ep_0 > 0$ so that  $\Dim_{\Lam}(\Im(P_{[0,\ep_0]}))$ and $\Dim_{\Lam}(\Im(P'_{[0,\ep_0]}))$ are finite, and that Assumption \ref{Assume} holds. Then, for $0 < \ep \leq \ep_0$,
$$
\Ind_{\Lam, \Lam'}(D_L^E, D_L^{E'}) \,\, = \,\,
\int_M \tr_s(k_{P_{[0, \ep]}}(x,x)) \, d\mu     \,\, - \,\, \int_{M'}   \tr_s(k_{P_{[0, \ep]}'}(x,x)) \, d\mu'.
$$
Thus,
$$
\int_M \tr_s(k_{P_0}(x,x)) d\mu     \,\, - \,\, \int_{M'}   \tr_s(k_{P_0'}(x,x)) \, d\mu' \,\, = \,\,
\int_{\cK} \AS (D_L ^E)_L  \, d\Lam \,\, - \,\,   \int_{\cK'} \AS (D_L ^{E'})_L  \, d\Lam'.
$$
}

For other results which show that restrictions on the spectral measures near zero of elliptic operators on foliations are necessary and sufficient to give index theorems,   see \cite{HL99, BH08, BHW14}.  

\medskip

Our second main result, an  extension of Theorem 3.2 of [GL83],  gives a condition which guarantees that the hypotheses of the foliation relative measured index theorem are satisfied. 

\medskip

\noi
{\bf Theorem \ref{measPR} } {\em Suppose the curvature operator $\maR_F^E$ is uniformly positive near infinity, that is, there is a compact subset $\cK \subset M$  and $\kappa_0 = \sup \{ \kappa \in \R \, | \, \maR^E_F -\kappa \Id \geq 0 \, \text{on} \, M\smallsetminus \cK\}$ is positive.  Then for $0 \leq \ep < \kappa_0$, $\Dim_{\Lam}(\Im(P_{[0,\ep]}))$  is finite.   More precisely,
$$
0 \; \leq \; \Dim_{\Lam}(\Im(P_{[0,\ep]}))   \,\, \leq\,\, \,\, \frac{(\kappa_0 - \kappa_1)}{(\kappa_0 - \ep)}  \int_{\cK}  \tr(k_{[0,\ep]}(x, x))  \,  d\mu \,\, < \,\, \infty,   
$$
where $\kappa_1  = \sup \{ \kappa \in \R \, | \, \maR^E_F -\kappa \Id \geq 0\, \text{on} \, M\}$.  In addition, Assumption \ref{Assume} holds.}

\medskip

Prime examples are spin foliations admitting leafwise metrics with positive scalar curvature (PSC) near infinity.

When $\dd \int_M \tr_s(k_{P_0}(x,x)) d\mu$ is finite, the $\Lam$-index  of  $D_L^E$, denoted $\Ind_{\Lam}(D_L^E)$, is well defined,  and is given by 
$$
\Ind_{\Lam}(D_L^E)  \; = \; \int_M \tr_s(k_{P_0}(x,x)) d\mu.
$$

Corollaries of Theorems \ref{Firstmain},  \ref{mainSection4}, and  \ref{measPR}  are the following.

\medskip

\noi
{\bf Theorem \ref{IMCOR1}}
{\em   Suppose that  $\maR_F^E$ is uniformly positive on $M\smallsetminus \maK$, so also  $\maR_{F'}^{E'}$ is uniformly positive on $M'\smallsetminus \maK'$.  Then 
$$
\Ind_{\Lam, \Lam'}(D_L^E, D_L^{E'}) \; = \; \Ind_\Lambda (D_L^E) - \Ind_{\Lambda'} (D_L^{E'}) \;  =  \;
\int_{\cK} \AS (D_L ^E)_L  \, d\Lam  -   \int_{\cK'} \AS (D_L ^{E'})_L  \, d\Lam'.
$$
}

\noi
{\bf Theorem \ref{IMCOR2}}
{\em Suppose that $E$ and $E'$ are two Clifford bundles over $M$ which are isomorphic off some compact subspace $\cK$ of $M$, and that $\cR^E_F$, so also  $\maR_{F}^{E'}$,  is uniformly positive on $M\smallsetminus \maK$.
Then 
$$
\Ind_{\Lam, \Lam}(D_L^E, D_L^{E'})  \,\, = \,\,  \Ind_\Lambda (D_L^E) - \Ind_\Lambda (D_L^{E'})
 \,\, = \,\, 
 \int_M (\AS (D_L)(\ch(E) - \ch(E'))_L  \, d\mu.
$$
}

For the next corollary, we say that the foliation $F$, and so also $F'$, is reflective if $\pa V$ is transverse to $F$,     so also $\pa V'$ is transverse to $F'$.  Then we can ``cut and paste" as in \cite{GL3} to get the compact  manifold $\what{M} = \cK \cup \cK'$ with foliation $\what{F}$,  transverse measure $\what{\Lam}$ and leafwise operator $\what{D}$. 

\smallskip

\noi
{\bf Theorem \ref{ExRelInd2}}  {\em Suppose that $F$ is reflective and $\cR^E_F$  is strictly positive off $\cK$, so also $F'$ is reflective and $\maR_{F'}^{E'}$ is  strictly positive off $\cK'$.  Then  
$$
\Ind_{\Lam, \Lam'}(D_L^E, D_L^{E'})   \,\, = \,\,  \Ind_{\what{\Lam}}(\what{D})   \,\, = \,\,
\int_{\cK} \AS (D_L ^E)_L  \, d\Lam  -   \int_{\cK'} \AS (D_L ^{E'})_L  \, d\Lam'.
$$
}

The previous construction  extends to the following more general situation, see again \cite{GL3}.  Assume $M \ssm \cK = V_+ \cup V_{\Phi}$  and  $M' \ssm \cK'  = V_+' \cup V_{\Phi}'$, where the unions are disjoint, that there is compatibility on the subsets $V_{\Phi}$ and $V_{\Phi}'$, that $F$ is reflective on $V_{\Phi}$, so $F'$ is reflective on $V_{\Phi}'$, and that  $\cR_F^E$ and $ \cR_{F'}^{E'}$ are strictly positive off $\cK$ and $\cK'$.  
Then we may cut and paste to get  the manifold 
$\what{M} = (M \ssm V_{\Phi}) \cup (M' \ssm V'_{\Phi})$, with the foliation
$\what{F}$,  the invariant  transverse measure $\what{\Lam}$, and the leafwise operator  $\what{D}_L^{\what{E}}$.
Because of the positivity off $\cK$, $\cK'$ and $\what{\cK} = \cK \cup \cK'$, all three operators $D_L^E$, $D_L^{E'}$ and $\what{D}_L^{\what{E}}$
 have finite  invariant  transverse measure indices, and we  have our final corollary, the  $\Phi$ relative index theorem, which will be useful in Section \ref{modspaces}.

\medskip

\noi
{\bf Theorem \ref{RelIndThm2}}  {\em Under the conditions above,
$$
\Ind_{\what{\Lam}}(\what{D}_L^{\what{E}}) \, = \,  
\Ind_{\Lam}(D_L^E) \,\, - \,\, \Ind_{\Lam'}(D_L^{E'}).
$$}

We now give a brief outline of the paper.
In Section \ref{results}, we give the specific setup we consider.  The techniques used in  \cite{GL3} of the proof of the classical relative index theorem are not available to us in general.  In particular, they consider a single non-compact manifold and an operator which is strictly positive off a compact subset.  This allows them to prove that the kernel of the operator is finite dimensional and there is a gap in the spectrum at $0$.  We consider a foliation $F$ of a non-compact manifold $M$ and a leafwise operator which is strictly positive off a compact subset $\cK \subset M$.  The intersection of a leaf $L$ of $F$ with $\cK$ may be a non-compact subset of $L$, considered as a manifold in its own right.  This causes problems, as the kernel of the operator on $L$ can then be infinite dimensional and there can be no gap in the spectrum at $0$.  To overcome these problems, we make the additional (rather strong) assumptions of bounded geometry and the existence of an invariant transverse measure.  This allows us to use the results and arguments of \cite{HL90}, extended in Section \ref{LDOs} from foliations of compact manifolds to the case of bounded geometry manifolds and foliations.  
 
 Section \ref{RelInd} is the heart of the paper.  It contains the definition of the relative measured index as well as the proof of the  foliation relative measured index theorem.   The proof uses mainly the theory of parametrices, analysis of  Schwartz kernels of operators, as well as the Spectral Mapping Theorem. 
 
 Section \ref{TIMAFTS}  contains the proof of Theorem \ref{measPR} and its corollaries.  The proof of Theorem \ref{measPR}  involves applying the leafwise Bochner identity to $k_{P_{[0,\ep]}}$. 

In Section \ref{modspaces} we define an invariant for pairs of PSC metrics on spin foliations as in \cite{GL3}, and show that if it is non-zero, then the metrics are not in the same path connected component of the space of PSC metrics on $F$.   We calculate this invariant for a large collection of spin foliations, and show that the space of  PSC metrics on each of these foliations has infinitely many path connected components. 

In this paper, we work leafwise on $M$ rather than on, say, the holonomy groupoid of $F$, since that would require us to assume the graph of $F$  is Hausdorff. This introduces some extra technicalities that we have to deal with. When the monodromy groupoid of  $F$ is Hausdorff while the holonomy groupoid is not, one can lift all the data to the monodromy covers and state the similar expected results there.  However the equivalence between our results here and the ones on the monodromy covers is not insured in general. Indeed, even with a single leaf whose fundamental group is not torsion free, some defect invariants can show up, see \cite{Be20}. Moreover, as in the classical  index theory for closed foliated manifolds, one may associate with the relative index data of  the present paper a higher index class, now living in the $K$-theory of a relative $C^*$-algebra and which does not need the existence of the holonomy invariant measures. So  our results here compute the image of this index class under a group morphism associated with the compatible pair $(\Lam, \Lam')$.

Finally, note that the results of this paper can be extended to the category of ``bounded geometry foliated spaces'' by adapting the constructions of \cite{MS06}. 
\medskip

\noindent
{\em Acknowledgements.}  It is a pleasure to thank Stephan Stolz for showing how to construct the manifolds we use in the examples in Section \ref{modspaces}.  

MTB  wishes to thank the french National Research Agency for support via the project ANR-14-CE25-0012-01 (SINGSTAR).

JLH wishes to thank the Simons Foundation for a Mathematics and Physical Sciences-Collaboration Grant for Mathematicians, Award Number 632868.

Both authors would like to thank the referee for cogent remarks which improved our paper.

\section{Preliminaries}\label{results}
Denote by  $M$ a  non-compact complete Riemannian manifold of dimension $n$, and by $F$  an oriented foliation (with the induced metric) of $M$ of dimension $p$, (until further notice, we assume that $p$ is even), and codimension $q = n - p$.    The metric on the leaves of $F$ induces a leafwise volume form denoted $dx_F$.  

The tangent and cotangent bundles of $M$ and $F$ are denoted $TM, T^*M, TF$ and $T^*F$.  A leaf of $F$ is denoted by $L$.  If $E$ is a bundle over $M$, the smooth sections are denoted by $C^{\infty}(E)$ and those with compact  support $C_c^{\infty}(E)$.  The smooth functions on $M$ are denoted by  $C^{\infty}(M)$ and those with compact  support by $C_c^{\infty}(M)$.  If $E$ carries a metric $h_x:E_x\otimes E_x \to \R$, the inner product map is abusively denoted
$$
\lan \cdot \ran: C^{\infty}(E \otimes  E) \to C^{\infty}(M),\text{ so it is given by } \langle \varphi\rangle (x):= h_x (\varphi (x)).
$$
In particular,  if $\varphi_1, \varphi_2 \in C^{\infty}(E)$, then 
$$
\lan \varphi_1\otimes \varphi_2  \ran(x) \,\, = \,\, h_x (\varphi_1 (x)\otimes \varphi_2(x)),\text{ also denoted }  \lan \varphi_1(x) , \varphi_2 (x)  \ran.
$$

We assume that both $M$ and $F$ are of bounded geometry, that is, the injectivity radius on $M$ and on all the leaves of $F$ is bounded below, and the curvatures and all of their covariant derivatives  on $M$ and on all the leaves of $F$ are uniformly bounded (the bound may depend on the order of the derivative).  Simple examples of one dimensional foliations on $\R^2$ show that bounded geometry on $M$ does not imply bounded geometry on the leaves of $F$.  We further assume that any connection or any metric on $E$ is uniformly bounded.   See \cite{Shubin} for material about bounded geometry bundles and their properties.

Let  ${\cU}$ be a good cover of $M$ by foliation charts as defined in \cite{HL90}.  In particular, denote by $\D^p (r)=\{x \in \R^p,  ||x||  <r\}$, and similarly for $\D^q(r)$.   An open locally finite cover  $\{(U_i, \psi_i)\}$ of $M$  by  foliation coordinate charts $\psi_i: U_i \to  \D^p(1) \times \D^q(1) \subset \R^n$  is a good cover for $F$ provided that
\begin{enumerate}
 \item 
For each $y \in \D^q(1), P_y = \psi_i^{-1}(\D^p(1) \times  \{y\})$ is contained in a leaf of $F$. $P_y$ is called a plaque of $F$.
\smallskip
\item 
If $\overline{U}_i \cap  \overline{U}_j \neq \emptyset$,  then $U_i \cap  U_j \neq \emptyset$, and $U_i \cap  U_j$ is connected.
\smallskip
\item 
Each $\psi_i$ extends to a diffeomorphism $\psi_i: V_i \to \D^p(2) \times \D^q(2)$, so that the cover $\{(V_i, \psi_i)\}$  satisfies $(1)$ and $(2)$, with $\D^p(1)$ and $\D^q(1)$ replaced by $\D^p(2)$ and $\D^q(2)$.
\smallskip
\item 
Each plaque of $V_i$  intersects at most one plaque of $V_j$  and a plaque of $U_i$ intersects a plaque of $U_j$ if and only if the corresponding plaques of $V_i$ and $V_j$ intersect.
\smallskip
\item
There are global positive upper and lower bounds on the norms of each of the derivatives of the $\psi_i$.
\end{enumerate}
Bounded geometry foliated manifolds always admit good covers. 

\medskip

When we mention measurable in this paper, that means borelian, i.e.\ measurable with respect to the Borel $\sigma$-algebra generated by the open subspaces for the underlying topology.  

\medskip

For each $U_i \in {\cU}$, let $T_i\subset U_i$ be a transversal (e.g. $T_i =  \psi_i^{-1}(\{0\} \times  \D^q(1))$) and set $T=\bigcup\,T_i$.  We may assume that the closures of the $T_i$ are disjoint.  Let $(U_i,T_i)$ and $(U_j,T_j)$ be elements of $\,\maU$, and $\gamma_{ij\ell}:[0,1] \to M$ be a path whose image is contained in a leaf with   $\gamma_{ij\ell}(0)\in T_i$ and $\gamma_{ij\ell}(1) \in T_j$. Then $\gamma_{ij\ell}$  induces a local homeomorphism $h_{\gamma_{ij\ell}}:  T_i \to T_j$, with domain $D_{\gamma_{ij\ell}}$ and range $R_{\gamma_{ij\ell}}$.   The  space $\cA^0_c(T)$ consists of all  uniformly bounded  measurable functions on $T $ which have compact support in each $T_i$.  The measurable Haefliger functions for $F$, denoted  $\cA^0_c(M/F)$, consists of elements in the quotient of $\cA^0_c(T)$ by the  vector subspace $W$ generated by elements of the form $\alpha_{ij\ell}-h_{\gamma_{ij\ell}}^*\alpha_{ij\ell}$ where $\alpha_{ij\ell} \in \cA^0_c(T)$ has support contained in $R_{\gamma_{ij\ell}}$.  We need to take care as to what this  means.  Members of $W$ consist of possibly infinite sums of elements of the form $\alpha_{ij\ell}-h_{\gamma_{ij\ell}}^*\alpha_{ij\ell}$, with the following restrictions: each member of $W$ has a bound on the leafwise length of all the $\gamma_{ij\ell}$ for that member, and each $\gamma_{ij\ell}$ occurs at most once.   Note that these conditions plus bounded geometry imply that for each member of $W$,  there is $n \in \N$ so that the number of elements of that member having $D_{\gamma_{ij\ell}}$ contained in any $T_i$ is less than $n$,  and that each $U_i$ and each $U_j$  appears at most a uniformly bounded number of times.  The projection map is denoted 
$$
[ \cdot ]: \cA^0_c(T) \,\,  \to \,\,  \cA^0_c(M/F).
$$

\medskip

Denote by  $\cA^{p}_b(M)$  the space of leafwise $p$-forms on $M$ which are leafwise smooth, transversely measurable and uniformly bounded.  As the bundle $TF$ is oriented, there is a continuous open surjective linear map, called integration over $F$,
$$
\dd \int_F:\cA^{p}_b(M) \to \cA^0_c(T).
$$
This map is given by choosing a partition of unity $\{\phi_i\}$ subordinate to the cover $\cU$, and setting 
$$
\int_F \omega \,\, = \,\, \sum_i \int_{U_i} \phi_i \omega.
$$
It is a standard result,  \cite{Ha80},  that the image of this differential form,  $\dd \Big[ \int_F \omega \Big]   \in \cA^0_c(M/F)$ is independent of the partition of unity.  
 
Note that $\dd \int_{U_i}$ is integration over the fibers of the projection $U_i \to T_i$, and
 that each integration $\omega \to \dd \int_{U_i} \phi_i \omega$ is essentially integration over a compact fibration, so $\dd \int_F$ satisfies the Dominated Convergence Theorem on each $U_i \in \cU$. 

\medskip

A graph chart $U_i \times_{\gamma_{ij\ell}} U_j \subset M \times M$,  is  a subset of the form
$$
U_i \times_{\gamma_{ij\ell}} U_j  \,\, = \,\,\bigcup_{z \in D_{\gamma_{ij\ell}}}  P_z \times P_{\gamma_{ij\ell}(z)}.
$$
It has a natural structure as a 2p+q dimensional manifold. 

For a real or complex bundle $E \to M$,  the external tensor product bundle  
$E \boxtimes E^*  \to M\times M$  restricts to a smooth bundle over $U_i \times_{\gamma_{ij\ell}} U_j$.  We denote the leafwise smooth, transversely measurable, bounded sections $k(x,y)$ with compact support of this bundle by $\Gam_c(U_i \times_{\gamma_{ij\ell}} U_j, E)$.   We extend them to all of $M \times M$ by  setting $k(x, y) = 0$ if $(x, y) \notin U_i \times_{\gamma_{ij\ell}} U_j$.   
\begin{defi}The  space  $\Gam_s(F,E)$ consists of sections $k$ of  $E \boxtimes E^*$, called kernels, such that $k$ is a (possibly infinite) sum $k = \sum_{ij\ell} k_{ij\ell}$,  with each $k_{ij\ell} \in \Gam_c (U_i \times_{\gamma_{ij\ell}}U_j, E)$.  For each $k$, we require that there is a bound on the leafwise length on its $\gamma_{ij\ell}$, and that each index $ij\ell$ occurs at most once.   Thus each $U_i$ and each $U_j$  appears at most a  bounded number of times, so the sum converges locally uniformly and in particular pointwise.   We further require that for each $k$, each of its leafwise derivatives in the local coordinates given by the good cover is uniformly bounded, with the bound possibly depending on the particular derivative. 
\end{defi}  

Denote by $E_L$ the restriction of $E$ to the leaf $L$.  
\begin{remark}
Recall the algebra $U\Psi^{-\infty} (L, E |_L)$ defined in \cite{Shubin}, Section A1.3, Definition 3.1.  Note  that elements of $\Gam_s(F,E)$ are  measurable families of elements of $U\Psi^{-\infty} (L, E |_L)$ with the bounds being uniform over $M$.
\end{remark}

If $k \in \Gam_s(F,E )$, it defines a leafwise operator 
$$
k: L^2(E_L) \to L^2(E_L)    \quad \text{ by } \quad
k(s)(x) \,\, = \,\, \int_L (k\,|_{L \times L})(x,y)s(y) \, dy_F.
$$
Because of the bounded geometry and the restriction on the lengths of the $\gamma_{ij\ell}$, the  operator corresponding to  $k \in  \Gam_s(F,E )$  has finite propagation, is leafwise smoothing, uniformly bounded, and transversely measurable.  See Theorems 2.3.1 and 2.3.2 of \cite{HL90}. 

Recall the notion of a super, that is $\Z_2$ graded, operator $A$.  Then the space $\cH$  which $A$ acts on splits as   $\cH = \cH^+ \oplus \cH^-$.  $A$ is an even operator if $A:\cH^{\pm} \to \cH^{\pm}$, and  an odd operator if  $A:\cH^{\pm}  \to \cH^{\mp}$.  If $A$ is an even (super) operator, its super trace is denoted 
$$
\tr_s(A) \,\, = \,\, \tr(A \, |_{\cH^+})  -   \tr(A \, |_{\cH^-}).
$$
If  $k \in \Gamma_s (F, E)$ and $x,y \in L$, then $k(x,y)$ is a linear operator from $E_y \to E_x$, the fibers over $y$ and $x$.  If it is an even operator, we set
$$
\tr_s (k(x, x))  \; = \; \tr \left(k(x, x)\vert_{E^+_x}\right) - \tr \left(k(x, x)\vert_{E^-_x}\right).
$$

\begin{definition}  The trace and  Haefliger trace of $k \in \Gamma_s(F,E)$  are given by
$$
\tr(k) \,\, = \,\,   \int_F  \tr(k(x,x))  \, dx_F \,\, \in \,\, \cA^0_c(T) \,\, \text{and} \,\, \htr(k) \,\, = \,\,   \Big[  \int_F  \tr(k(x,x))  \, dx_F \Big] \,\, \in \,\, \cA^0_c(M/F).
$$
If $k$ is even, its super-trace  and Haefliger super-trace are given by
$$
\tr_s(k) \,\, = \,\,   \int_F  \tr_s(k(x,x))  \, dx_F \,\, \in \,\, \cA^0_c(T) \,\, \text{and} \,\,\htr_s(k) \,\, = \,\, \Big[  \int_F  \tr_s (k(x, x)) \, dx_F \Big]    \,\, \in \,\, \cA^0_c(M/F).
$$
\end{definition} 

We end this section by recalling the following.

\begin{theorem}{\cite{HL90},Theorem 2.3.6.}\label{trprop}  Suppose $k_1, k_2 \in \Gamma_s(F,E )$  are super operators.   If both are even, 
$$\htr_s(k_1 \circ k_2) \,\, =  \,\, \htr_s(k_2 \circ k_1),$$
and if both are odd,
$$\htr_s(k_1 \circ k_2)  \,\, =  \,\, -\htr_s(k_2 \circ k_1),$$
in $\cA_c^0(M/F)$. 
\end{theorem}

Note that while the functions $\tr_s(k_1 \circ k_2)$ and $\tr_s(k_2 \circ k_1)$ also exist in $\cA_c^0(T)$, in general they are not equal. 

\begin{proof}
We do only the even case.  Because of the limit on the leafwise length of the $\gamma_{ij\ell}$, for each $i,j$, there are only finitely many $k_{1,ij\ell}$ in the sum making up  $k_1$.  Similarly,  for each $r,s$,  there are only finitely many $k_{2,rst}$ in the sum making up  $k_2$.   As
$$
\htr_s(\sum_{ij\ell} k_{1,ij\ell} \circ \sum_{rst}k_{2,rst}) =  \sum_{ij\ell}\sum_{rst} \htr_s(k_{1,ij\ell} \circ k_{2,rst}),
$$
we may assume $k_1 = k_{1,ij\ell} \in \Gam_c(U_i \times_{\gamma_{ij\ell}} U_j,E)$ and  $k_2 = k_{2,rst} \in \Gam_c(U_r \times_{\gamma_{rst}} U_s,E)$ with $U_j\cap U_r \neq \emptyset$ and $U_s\cap U_i \neq \emptyset$, since otherwise $\htr_s(k_{1,ij\ell} \circ k_{2,rst}) = 0$.
 
Since $\htr_s$ does not depend on the partition of unity, we may assume $(\phi_i \times \phi_j)k_1 = k_1$, and 
$(\phi_r \times \phi_s)k_2 = k_2$, so the partition of unity will play no role here. 
  
Suppose $z \in T_i$, with $x \in P_z$ and $y \in P_{\gamma_{ij\ell}(z)}$,  and $x \in U_i \cap U_s$ and $y \in U_j \cap U_r$.   
Note that, in order to get something non-trivial,  we must have $\gamma_{rst}\gamma_{ij\ell}(z) = z$.  Then
$$
\tr_s(k_1 \circ k_2) (z) \,\, = \,\,  \int_{P_z} \int_{L_z} \tr_s(k_1(z,x,y) k_2(\gam_{ij\ell}(z),y,x)) \, dy_F dx_F\,\, = \,\,  
$$
$$
 \int_{P_z} \int_{P_{\gamma_{ij\ell}(z)}}  \hspace{-0.6cm}  \tr_s(k_1(z,x,y) k_2(\gam_{ij\ell}(z),y,x))  \,  dy_F dx_F,
$$
since $k_1(x, y) = 0$ unless $x \in P_z$ and $y \in P_{\gamma_{ij\ell}(z)}$. 

Similarly,
$$
\tr_s(k_2 \circ k_1) (\gamma_{ij\ell}(z)) \,\, = \,\,  \int_{P_{\gamma_{ij\ell}(z)}}\int_{\gamma_{rst}\gamma_{ij\ell}(z)}   \hspace{-1.0cm}  \tr_s( k_2(\gam_{ij\ell}(z),y,x)k_1(\gam_{rst}\gam_{ij\ell}(z),x,y)) dx_F dy_F,
$$
which has exactly the same value as $\tr_s(k_1 \circ k_2) (z)$, but at $\gamma_{ij\ell}(z) \in T_j$, since $\gamma_{rst}\gamma_{ij\ell}(z) = z$.   Thus $\tr_s(k_1 \circ k_2) \,|_{T_i} = h^*_{ \gamma_{ij\ell}}(\tr_s(k_1 \circ k_2) \,|_{T_j})$, so their images in $\cA^0_c(M/F)$ are the same.
\end{proof}

\section{Overview of leafwise Dirac operators} \label{LDOs}

In this section we give extensions of some results from \cite{HL90}, see also \cite{H02}, to our more general setting.
The proofs for the case considered here are essentially the same as in \cite{HL90}.  The main things to notice are these.
\begin{enumerate}
\item  The space denoted $C^{\infty}_0(F,E)$ is replaced by the space  $\Gam_s(F,E)$. 
\item The bounds coming from the compactness of $M$ still hold due to our assumption of bounded geometry.
\item  The geometric endomorphism is just the identity map, the invariant  transverse measure is ignored, and $\dd \int_M$ is replaced by $\dd \int_F$.
\item All the operators considered here are transversely measurable.  
\end{enumerate}

\medskip

A  leafwise Dirac operator $D_L^E$ consists of a Dirac bundle $E$, that is a Clifford bundle over the Clifford algebra of $T^{*}F$, and a Hermitian connection $\nabla$ on $E$, compatible with Clifford multiplication, so that the 
operator 
$$
D_L^E:C^{\infty}_{c}(E) \to C^{\infty}_{c}(E)
$$
is given by the composition
$$
C^{\infty}_{c}(E)   \stackrel{\nabla}{\to}  C^{\infty}_{c}(T^{*}M \otimes E)\stackrel{\rho}{\to}  C^{\infty}_{c}(T^{*}F \otimes E)
\stackrel{m}{\to}    C^{\infty}_{c}(E),
$$
where $\rho$ is the restriction and  $m$ is Clifford multiplication.   For more details, see \cite{LM}.   In particular, if we identify $T^{*}F$ and $TF$ using the metric, then locally 
$$
D_L^E(s) \,\, = \,\, \sum_{j = 1}^p e_j \cdot \nabla_{e_j} s,
$$
where $e_1,...,e_p$ is a local orthonormal basis of $TF$, and $e_j \cdot$ is Clifford multiplication by $e_j$.
All the classical complexes (de Rham, Signature, Dolbeault, and Spin) give rise to leafwise Dirac operators  provided $F$ supports the necessary geometric structures for these complexes to be defined.   

Since the leaves $L$ are complete, $D_L^E$ is essentially self adjoint, \cite{Ch73}.  Thus any bounded Borel function $g$ on $\R$ applied to $D_L^E$ yields a well defined bounded leafwise operator $g(D_L^E): L^{2}(E_L) \to L^{2}(E_L)$.   
The operator we are interested in is $e^{-t(D_L^E)^2}$.   
Unfortunately, its Schwartz kernel $k_{e^{-t(D_L^E)^2}}$ is generally not in $\Gam_s(F,E )$.  However, we do have,
\begin{theorem}\label{SCF&D1}{\cite{HL90},Theorem 2.3.7.} Suppose that $g$ is a Schwartz function whose Fourier transform is in $C^{\infty}_c(\R)$, and that $B$ is a differential operator on $E$ along $F$ with smooth bounded coefficients.  Then  the Schwartz kernels of $g(D_L^E)$, $Bg(D_L^E)$, and $g(D_L^E)B$ are in $\Gam_s(F,E )$. 
\end{theorem}

Schwartz functions can be approximated by elements in $C^{\infty}_c(\R)$, and using the Fourier inversion formula to define operators works well in our setting. For more on this see Section \ref{RelInd}.  If $g$ is a Schwartz function, then estimates by Schwartz functions whose Fourier transforms are in $C^{\infty}_c(\R)$, as given in \cite{HL90}, along with bounded geometry, show that $k_{g(D_L^E)}(x,y)$ is uniformly bounded on $M \times M$.  The same holds for 
the Schwartz kernels of $Bg(D_L^E)$, and $g(D_L^E)B$.  In particular, $\tr(k_{g(D_L^E)}(x,x))$, $\tr(k_{Bg(D_L^E)}(x,x))$, and $\tr(k_{g(D_L^E)B}(x,x))$ are uniformly bounded on $M$.  Thus we get,
\begin{theorem}{\cite{HL90},Theorem 2.3.8.} \label{SCF&D2} 
Suppose that $g$ is a Schwartz function. Then $\tr(g(D_L^E))$ and $\htr(g(D_L^E))$ exist.  If $B$ is a differential operator on $E$ along $F$ with smooth bounded coefficients,  then $\tr(Bg(D_L^E))$, $\htr(Bg(D_L^E))$, $\tr(g(D_L^E)B)$, and $\htr(g(D_L^E)B)$ exist.  The same holds for the super traces provided the operators are super operators.
\end{theorem} 

Since  for  $t > 0$, $e^{-tx^2}$ is a Schwartz function,  $\tr_s(e^{-t(D_L^E)^2})$ and $\htr_s(e^{-t(D_L^E)^2})$ exist.  Classical local results give the following.  See \cite{Ge86}. 
Denote by $\AS (D_L^E)$ the characteristic differential form associated to $D_L^E$ by the local Atiyah-Singer Index Theorem, \cite{ABP73,Gi73}.  Denote its restriction to $C^{\infty}(\bigwedge T^*F)$ by $\AS (D_L^E)_L$.
\begin{theorem}\label{zerolim}$\lim_{t \to 0}  \tr_s(k_{e^{-t(D_L^E)^2}}(x,x)) = \AS (D_L^E)_L(x)$ uniformly on $M$.   Thus,
$$
 \dd \lim_{t \to 0} \tr_s(e^{-t(D_L^E)^2})    \,\, = \,\,    \int_F \AS (D_L^E)_L \quad \text{in \, $\cA^0_c(T)$},  \,\, \text{and}
$$  
$$
\lim_{t \to 0} \htr_s(e^{-t(D_L^E)^2})    \,\, = \,\,  \Big[  \int_F \AS (D_L^E)_L \Big] \quad \text{in \, $\cA^0_c(M/F)$}.
$$  
\end{theorem}  

Denote the Schwartz kernel of the graded projection onto the leafwise kernel of  $(D_L^E)^2$ by $k_{P_0}$.  The rest of Section 2.3 of {\cite{HL90} is taken up with the technicalities of proving the following.
\begin{theorem} {\cite{HL90}, Theorem 2.3.11.}\label{schwartz3}
$\dd  \lim_{t \to \infty}   \tr_s(k_{e^{-t(D_L^E)^2}}(x,x)) \,\, = \,\,  \tr_s(k_{P_0}(x,x))$  pointwise on $M$.
 \end{theorem}
In general, this convergence is not uniform. 
However, since $\tr_s(k_{e^{-t(D_L^E)^2}}(x,x))$ is uniformly bounded on $M$,  and $\dd \int_{U_i}$ is essentially integration over a compact fibration, it follows that for any \underline{finite} collection of $(U_i,T_i)$, 
$$
 \lim_{t \to \infty} \sum_i \int_{U_i} \phi_i(x)  \tr_s(k_{e^{-t(D_L^E)^2}}(x,x)) \,\, = \,\, \sum_i \int_{U_i}  \phi_i(x) \tr_s(k_{P_0}(x,x))  
\,\, \text{ in } \, \sum_i \cA^0_c(T_i) \subset \cA^0_c(T). 
$$

Next, we have,
\begin{theorem}{\cite{HL90}, Theorem 5.1.}\label{indept} 
The element $\htr_s(e^{-t(D_L^E)^2}) \in \cA^0_c(M/F)$ is independent of $t$. 
\end{theorem}
The proof involves taking limits of approximations by elements of $\Gam_s(F,E)$.

\medskip

Finally, we have two useful results from the Spectral Mapping Theorem.
\begin{prop}\label{ptwseconv} Suppose that the sequence of  bounded Borel functions $f_n(z)$ converges pointwise to $f(z)$, and for $\ell$ sufficiently large, $||(1 + z^2)^{\ell/2}f_n||_{\infty}$ is a bounded sequence. Then the Schwartz kernel $k_{f_n(D_L^E)}$ converges to $k_{f(D_L^E)}$ pointwise.
\end{prop}
\begin{proof} For $x \in M$ and $v \in E_x$, the fiber over $x$,  with $||v|| = 1$, denote by $\delta_x^v$ the distributional section of $E_L$ given by $\lan \delta_x^v, \sigma \ran = \lan \sigma(x), v \ran$.
Because of bounded geometry, there is $\ell$ sufficiently large, which depends only on the dimension of $F$, so that the Sobolev $-\ell$ norm 
$$
||\delta_x^v||_{-\ell} \; = \; ||(1 + (D_L^E)^2)^{-\ell/2}\delta_x^v||_0
$$
of  $\delta_x^v$  is uniformly bounded over all $x$ and $v$.  The norm $|| \cdot ||_0$ is the $L^2$ norm. 

The Spectral Mapping Theorem says that if a sequence of  bounded Borel functions $g_n$ converges pointwise to $g$ and the sequence $||g_n||_{\infty}$ is bounded, then $g_n(D_L^E)$ converges to $g(D_L^E)$ strongly, that is for any 
element $v$, $\lim_{n \to \infty}||g_n(D_L^E)(v) - g(D_L^E)(v)||_0 = 0$.

Now,  $||(k_{f_n(D_L^E)}  -  k_{f(D_L^E)} )(x,y)||$ is bounded by a finite sum of elements of the form 
$$
|  \lan (k_{f_n(D_L^E)}  -  k_{f(D_L^E)} )(x,y) )(w), v \ran |,
$$
 where $w \in E_y$, and $v \in E_x$ and both have norm $1$.
But, we have
$$
| \lan (k_{f_n(D_L^E)}  -  k_{f(D_L^E)} )(x,y) )(w), v \ran | \; = \; 
|\lan (f_n(D_L^E) - f(D_L^E))(\delta_y^w), \delta_x^v\ran| \; \leq \;
$$
$$
||(f_n(D_L^E) - f(D_L^E))(\delta_y^w)||_{\ell}   ||\delta_x^v||_{-\ell} \; = \;
||((1 + (D_L^E)^2)^{\ell/2}(f_n(D_L^E) - f(D_L^E)))(\delta_y^w)||_0 ||\delta_x^v||_{-\ell} \; = \;
$$
$$ 
||((1 + z^2)^{\ell/2}(f_n - f))(D_L^E)(\delta_y^w)||_0 \, ||\delta_x^v||_{-\ell}.
$$
As $f_n$ converges pointwise to $f$, $(1 + z^2)^{\ell/2}f_n$ converges pointwise to  $(1 + z^2)^{\ell/2}f$. Since  $||(1 + z^2)^{\ell/2}f_n||_{\infty}$ is bounded, $((1 + z^2)^{\ell/2}f_n)(D_L^E)$ converges strongly to  
$((1 + z^2)^{\ell/2}f)(D_L^E)$, so 
$$
\lim_{n \to \infty}||((1 + z^2)^{\ell/2}(f_n - f))(D_L^E)(\delta_y^w)||_0  \; = \; 0.
$$
\vspace{-1.0cm} 

\end{proof}

\begin{prop}\label{bddedconv} 
Suppose that $f_t \to f$ as $t \to 0$ in the Schwartz topology on the Schwartz functions.  Then,
$$
\lim_{t \to 0} k_{f_t(D_L^E)} \; = \; k_{f(D_L^E)}  \quad \text{uniformly}.
$$
\end{prop}
\begin{proof}
By the Spectral Mapping Theorem $||g(D_L^E)|| \leq \sup_{z \in \R} |g(z)|$.   Now
$f_t \to f$ as $t \to 0$ in the Schwartz topology gives that for all $n \geq 0$, 
$
\sup_{z \in \R}|z^n (f_t(z) - f(z))| \to 0 \quad \text{as} \quad t \to 0.
$
As above,  we have,
$$
| \lan (k_{f_t(D_L^E)}  -  k_{f(D_L^E)} )(x,y) )(w), v \ran | \;\leq \; 
||(f_t(D_L^E) - f(D_L^E))(\delta_y^w)||_{\ell}   ||\delta_x^v||_{-\ell} \;\leq \;
$$
$$
||f_t(D_L^E) - f(D_L^E)||_{-\ell,\ell}   ||\delta_y^w||_{-\ell} || \delta_x^v ||_{-\ell}, 
$$
where $||f_t(D_L^E) - f(D_L^E)||_{-\ell,\ell} $ is the norm of operator from the $-\ell$ Sobolev space to the  $\ell$ Sobolev space.  Now 
$$
||f_t(D_L^E) - f(D_L^E)||_{-\ell,\ell}  \; \leq \;  \sup_{z \in \R}|(1 + z^2)^{\ell}(f_t(z) - f(z))|,
$$
which goes to zero as $t\to 0$, independently of $\delta_y^w$ and $\delta_x^v$.
\end{proof}

\section{The foliation relative measured index theorem }\label{RelInd} 

In this section, we assume that we have two foliated manifolds $(M, F)$ and $(M', F')$ as above,  and two  $\Z_2$ graded odd leafwise Dirac operators $D_L^E$ and $D_L^{E'}$ acting on Clifford bundles $E \to M$ and $E' \to M'$, with Clifford compatible Hermitian connections $\nabla$ and $\nabla'$.    We further assume that there are compact subspaces $\cK = M\smallsetminus V$ and $\cK' = M'\smallsetminus V'$ of $M$ and $M'$ with a  bundle morphism $\Phi=( \phi, \varphi )$ from $E\to V$ to $E'\to V'$.  We assume that $\varphi:V \to V'$ is an isometry with $\varphi^{-1}(F') = F$,  that $\phi :E\vert_V \to E'\vert_{V'}$ is an isomorphism, and that $\phi^*(\nabla' \, |_{V'}) = \nabla \, |_V$. Thus,  the well defined (since they are differential operators) restrictions of $D_L^E$ and $D_L^{E'}$ to the sections over $V$ and $V'$ agree through $\Phi$, i.e.
$$
(\Phi^{-1})^*\circ D_L^E \circ \Phi^* \,|_{V'} = D_L^{E'} \,|_{V'}. 
$$
Without loss of generality, we may assume that $\cK$ and $\cK'$ are the closures of open subsets.

We may assume that the good open covers $\cU$ and $\cU'$ on $M$ and $M'$  are $\varphi$ compatible on $V$ and $V'$.  That is, 
$$
\{ U_i \in \cU \, | \, \overline{U_i} \cap \cK = \emptyset \} \,\, =  \,\,  \{ \varphi^{-1}(U_i') \, | U_i' \in \cU',  \overline{U_i'} \cap \cK' = \emptyset \}.
$$
Set 
$$
\cU_V \,\, = \,\, \{ U_i \in \cU \, | \, \overline{U_i} \cap \cK = \emptyset \} \quad \text{and} \quad    \cU'_{V'} \,\, =  \,\,  \{ U_i' \in \cU' \, |   \overline{U_i'} \cap \cK' = \emptyset \},
$$
and 
$$
T_V \,\, = \,\, \{ T_i \in \cU \, | \, \overline{U_i} \cap \cK = \emptyset \} \quad \text{and} \quad    T'_{V'} \,\, =  \,\,  \{ T_i' \in \cU' \, |   \overline{U_i'} \cap \cK' = \emptyset \}.
$$  
Thus  $\varphi^*: \cA_c^0 (T'_{V'}) \ \to \cA_c^0 (T_V)$ is an isomorphism from the functions
supported on the transversals in  $\cU'_{V'}$ to the functions supported on the transversals  in  $\cU_{V}$. Denote by $T_{\cK} = T \ssm T_V$, the transversals which are \underline{not} in $\cU_{V}$, and similarly for  $T'_{\cK'}$, both of which are relatively compact. 

Finally, we assume that we have  invariant  transverse measures $\Lam$ and $\Lam'$  on $(M, F)$ and $(M', F')$, which are $\varphi$ compatible on $\cU_V$ and $\cU'_{V'}$, that is, for any $\alp' \in  \cA^0_c (T') $ and $(U'_i,T'_i) \in \cU'_{V'}$,
$$
\int_{\varphi^{-1}(T'_i)}  \varphi^*(\alp' |_{T'_i}) \, d\Lam \ \,\, = \,\,   \int_{T'_i}\alp' |_{T'_i} \,  d\Lam'.
$$
Recall that $\Lam$ is a measure on each $T_i$ so that if $f_{ij}:T_i \to T_j$ is a local diffeomorphism induced by the holonomy of $F$, then it preserves the measure.  By the obvious extension, $\Lam$ induces a Borel measure on any transversal to $F$  which is $\sigma$-finite, i.e.\ for any compact transversal $\what{T}$, $\dd \int _{\what{T}} 1  \ d \Lam$ is finite.  The leafwise measure $dx_F$ and $\Lam$ combine to give a global measure denoted $d\mu$.  In particular, 
$$
\int_M  \bullet \,d\mu \,\, = \,\, \int_T \left[\int_F\bullet \,  dx_F\right] d\Lam.
$$
Similarly for $\Lam'$.

\medskip

Next, we introduce the $\varphi$-relative space of Haefliger functions, along with their relative integration against  $\Lambda$ and $\Lambda'$.  Denote by  $g: M\to [0, \infty)$ and $g': M'\to [0, \infty)$ two smooth exhaustions such that for any $s \geq s_0$ for some $s_0>0$, the open subspaces, with compact complements, $M(s)=\{g>s\}$ and $M'(s)=\{g'>s\}$ agree through $\varphi$,   that  is $ \varphi(M(s_0)) = M'(s_0)$ and $g |_{M(s_0)} = g'\circ \varphi  |_{M(s_0)}$.   For $s \geq s_0$, set 
$$
T_s \,\, = \,\, \{T_i \in T \, | \, T_i  \cap M(s) \neq   \emptyset \},
$$
and similarly for $T_s'$.

Recall the subspaces  $W  \subset \maA_c^0(T)$ and $W'  \subset  \maA_c^0(T')$  from the definition of the bounded measurable Haefiger functions $\maA^0_c(M/F)$ and  $\maA^0_c(M'/F')$ given in Section \ref{results}.  Suppose that 
$(\omega,\omega') \in W \times  W'$, with  $\omega = \sum_{(\alp, \gam)} \alpha - h^*_{\gam} \alpha$ and $\omega' = \sum_{(\alp', \gam')} \alpha' - h^*_{\gam'} \alpha'$.   For simplicity, we have dropped the subscripts.  The  vector subspace $W \times_{\varphi} W' \subset W \times  W'$  consists of elements $(\omega,\omega')$  which are $\varphi$ compatible.   This means that all but a finite number of the $(\alp, \gam)$ and $(\alp', \gam')$ are paired, that is 
$$
\alpha \; = \; \varphi^*(\alpha') \quad \text{and} \quad \gamma' \; = \; \varphi\circ \gamma, \quad \text{so} \quad 
\alpha - h^*_{\gam} \alpha \; = \;  \varphi^*( \alpha' - h^*_{\gam'} \alpha').
$$
\begin{definition}
Given  functions $\beta  \in \maA_c^0(T)$ and $\beta'\in\maA_c^0(T')$,  the pair $(\beta,\beta') $ is $\varphi$-compatible if  there exists $s \geq s_0$ so that  $\beta = \varphi^*(\beta') $ on $T_s$.

Set 
$$
\maA_c^0 (M/F, M'/F'; \varphi) \,\, = \,\, \{ (\beta, \beta') \in \maA_c^0(T)\times \maA_c^0(T') \, | \, (\beta,\beta')  \text{ is $\varphi$ compatible} \}\, / (W \times_{\varphi} W'). 
$$
\end{definition}  

\begin{definition}   For $[(\beta, \beta')] \in \maA_c^0 (M/F, M'/F'; \varphi)$, set 
$$  
\lan [(\beta, \beta')], (\Lam,\Lam') \ran \,\, = \,\,  
\lim_{s \to \infty} \left(\int_{T \ssm T_s}  \hspace{-0.5cm} \beta \, d\Lam - \int_{T' \ssm T_s'}  \hspace{-0.5cm}\beta' \, d\Lam' \right).
$$
\end{definition}
\noindent 
This is well defined because any representative $(\beta, \beta')$ is $\varphi$ compatible, so the right hand side is eventually constant.  In addition, every $(\omega, \omega') \in W \times_{\varphi} W'$ is $\varphi$ compatible, so satisfies
$$
\lim_{s\to \infty} \left(\int_{T \ssm T_s}  \hspace{-0.4cm}  \omega \, d\Lam - \int_{T' \ssm T_s'} \hspace{-0.5cm} \omega' \, d\Lam' \right) \,\, = \,\,   0.
$$
To see this, recall that there is a global bound on the leafwise length of the $\gam$ and $\gam'$ in $\ome$ and $\ome'$.  This, and the fact that there are only finitely many unpaired  $(\alpha,\gam)$ and  $(\alpha',\gam')$, insures that for large $s$, every unpaired $(\alpha,\gam)$ will have both $D_{\gamma}$ and $R_{\gamma} \subset T \ssm T_s$, so $\dd \int_{T \ssm T_s}  \hspace{-0.4cm}   \alpha - h^*_{\gam} \alpha \, d\Lam$  will be zero, and similarly for  every unpaired $(\alpha',\gam')$.  Those $(\alpha,\gam)$ and $(\alpha',\gam')$ which are paired and appear in the integration, will have  $D_{\gamma}$ and/or $R_{\gamma} \subset T \ssm T_s$  with corresponding $D_{\gamma'}$ and/or $R_{\gamma'} \subset T' \ssm T'_s$.   In both cases, their integrals will cancel.

Throughout the paper, we denote the leafwise Schwartz kernel of a leafwise operator $A$ by $k_A(x,y)$, which is a section of the external tensor product bundle $E \boxtimes E^*  \simeq E \boxtimes E$ over $M\times M$.  The restriction of $k_A$ to the diagonal in $M\times M$ is then a section of $E\otimes E^*\simeq E\otimes E$ over $M$.
Two sections $\varphi_1$ and $\varphi_2$ of $E$ give the section $\varphi_1 \boxtimes \varphi_2$ of $E \boxtimes E$, which acts on  a section $\varphi$ of $E$ by, 
$$
(\varphi_1 \boxtimes \varphi_2)(\varphi) \,\, = \,\,   \langle \varphi_2, \varphi \rangle  \varphi_1.
$$
The restriction of $\varphi_1 \boxtimes \varphi_2$ to the diagonal is denoted $\varphi_1 \otimes \varphi_2$, and it is then clear that, 
$$
\tr(\varphi_1 \otimes \varphi_2) = \langle \varphi_1, \varphi_2 \rangle,     \quad   \text{that is}\quad                                            \tr(\varphi_1 \otimes \varphi_2)(x) = \langle \varphi_1(x), \varphi_2(x) \rangle.
$$
This extends to all sections $k_A$, by first restricting to the diagonal, so we have the suggestive notation $k_A \lan x,x \ran = \tr(k_A(x,x))$.

As $D_L^E$ and $D_L^{E'}$  are  odd super operators, we have 
$$
D_L^E \,\, = \,\, \left(\begin{array}{cc}  0& (D_L^E)^- \\ (D_L^E)^+  & 0\end{array}\right) \text{ and }
D_L^{E'} \,\, = \,\, \left(\begin{array}{cc}  0 & (D_L^{E'})^- \\ (D_L^{E'})^+ & 0\end{array}\right).
$$

Let $Q$ and $Q'$ be leafwise parametrices for $D_L^E$ and $D_L^{E'}$, respectively.  That is, they are finite propagation odd operators, (which are zero in the $+$ to $-$ direction), with $\varphi$  compatible remainders
$$
S=\Id - Q(D_L^E)^+, \,\, R=\Id - (D_L^E)^+ Q, \,\, S' =\Id - Q'(D_L^{E'})^+, \text{ and } R'=\Id - (D_L^{E'})^+Q'.
$$
The remainders have Schwartz kernels which 
belong to $\Gam_s (F, E)$ and $\Gam_s (F', E')$ respectively, and they are $\varphi$ compatible.  
Thus each remainder has finite propagation,  and for $s \geq s_0$ sufficiently large,  they are identified by $\Phi = (\phi, \varphi)$.  
For example,
$$
(\Phi^{-1})^*\circ S \circ \Phi^* \,|_{M'(s)} = S' \,|_{M'(s)}, 
$$
for $s$ so large that $S$ sends sections supported on $M(s)$ to sections supported on $M(s_0)$,
and similarly for $M'(s)$.  The same is true for their squares with a possibly larger $s_0$.   It is easy to check that such parametrices always exist, and that, because of bounded geometry,  they satisfy the properties we need.  See \cite{Shubin}, and the proof of Theorem \ref{mainSection4} below. Set
$$
\Ind (D_L^E)  \,\, = \,\, \htr(k_{S^2}) - \htr(k_{R^2}) \in \cA_c^0(M/F),
$$
and similarly for $\Ind (D_L^{E'})$.  

\begin{definition}
The relative measured index of the  pair $(D_L^E, D_L^{E'})$ of leafwise  Dirac operators is
\vspace{-0.3cm}
$$
\Ind_{\Lam, \Lam'}(D_L^E, D_L^{E'}) \,\, = \, 
\lim_{s \to \infty} \left(\int_{M \ssm M(s)}  \hspace{-1.0cm} k_{S^2} \lan x,x \ran  - k_{R^2} \lan x,x \ran\, d\mu  \,\, - \,\, 
\int_{M' \ssm M'(s)}    \hspace{-1.2cm}  k_{{S'}^2} \lan x,x \ran   - k_{{R'}^2} \lan x,x \ran \, d\mu'  \right).
$$
\end{definition}

\begin{remark}\label{nonsquare} We could as well use $S$, $R$, $S'$ and $R'$ in place of  $S^2$, $R^2$, ${S'}^2$ and ${R'}^2$ above.  Both are convenient for different applications.  The RHS still makes sense without the squares and is unchanged. One just writes  $S^2=S - Q [(D_L^E)^+ S]$ and $R^2=R-[(D_L^E)^+ S] Q$ and similarly for ${S'}^2$ and ${R'}^2$, and shows that all the extra terms cancel out in the $s$ limit.  For a proof see the Appendix.
We use the squares here and the non-squares in the proof of Theorem \ref{mainSection4} below.

\end{remark}

\begin{proposition}\label{indepQ}
$\Ind_{\Lam, \Lam'}(D_L^E, D_L^{E'})$  does not depend on the choice of  $(Q,Q')$ with $\varphi$ compatible remainders.  
\end{proposition}

\begin{proof}
By an abuse of notation, we will replace the Schwartz kernels of operators by the operators themselves.  Suppose that $(Q_0,Q_0')$ and $(Q_1,Q_1')$ are  two pairs of parametrices  with $\varphi$ compatible remainders. Denote the remainders for $(D_L^E)^+$ by $S_0, R_0$ and $S_1, R_1$ respectively.   The expression 
$$
\htr((\Id - Q_i(D_L^E)^+)^2) - \htr((\Id - (D_L^E)^+ Q_i)^2) = \htr(S_i^2) - \htr(R_i^2),
$$
can be interpreted as  $\htr(e_i-f)$ for idempotents $e_i, f$ in $\Gam_s (F, E)\oplus (\C\Id_{E^+}\oplus \C\Id_{E^-})$ such that $e_i-f\in \Gam_s (F, E)$. More precisely, set
$$
e_i \,\, = \,\, \left(\begin{array}{cc}  S_i^2 &  Q_i (R_i+R_i^2) \\ R_i (D_L^E)^+ & \Id_{E^-} -R_i^2\end{array}\right) \text{ and }f \,\, = \,\,\left(\begin{array}{cc} 0 & 0 \\ 0 & \Id_{E^-} \end{array}\right)
$$
Notice that the Haefliger trace $\htr$ extends to a trace  $\widetilde\htr$ on 
$$
\{k+\left(\begin{array}{cc}\lambda_+ \Id_{E^+} & 0 \\ 0 & \lambda_-\Id_{E^-}\end{array}\right) \vert \; k\in \Gam_s (F, E)\text{ and }\lambda_\pm\in \C\},
$$ 
which is defined to be zero  on $\C\Id_{E^+} \oplus \, \C\Id_{E^-}$.   We thus have
$$
\htr(S_i^2) - \htr(R_i^2) =  \widetilde\htr(e_i).
$$

For  $t\in [0, 1]$, set
$$
Q_t = (1-t)Q_0 + tQ_1,
$$
which is  a one parameter family of  leafwise parametrices  from $Q_0$ to $Q_1$,  with  remainders $S_t=(1-t)S_0+tS_1$ and $R_t=(1-t)R_0+tR_1$.  Then  
$$
e_t \,\, = \,\,\left(\begin{array}{cc} S_t ^2&  Q_t (R_t+R_t^2) \\ R_t (D_L^E)^+ & \Id_{E^-} - R_t^2\end{array}\right)
$$
is a family of idempotents such that 
$$
\htr(S_t^2) - \htr(R_t^2) = \htr(e_t - f)=\widetilde\htr(e_t).
$$
Taking the derivative with respect to $t$, we get
$$
\frac{d}{dt} \left(\htr(S_t^2) - \htr(R_t^2) \right) = \widetilde\htr(\dot{e}_t) = \htr(\dot{e}_t) = \htr(e_t \dot{e}_t + \dot{e}_t e_t)= 2 \htr(e_t \dot{e}_t e_t).
$$
Since $e_t$ is an idempotent,  $e_t \dot{e}_t e_t=0$ for any $t$, and so $\htr(S_t^2) - \htr(R_t^2)$ is independent of $t$, that is, $\Ind (D_L^E)$ does not depend on the choice of $Q$.  

The same argument shows that $\Ind (D_L^{E'})$ does not depend on the choice of $Q'$.
An important point to note is that all of the elements in the argument for  $\Ind (D_L^{E})$ are $\varphi$ compatible with those in the argument for  $\Ind (D_L^{E'})$. This implies that the element 
$$
(\htr(S_t^2) - \htr(R_t^2), \htr({S'}_t^2) - \htr({R'}_t^2))  \; = \;[(\tr(S_t^2) - \tr(R_t^2), \tr({S'}_t^2) - \tr({R'}_t^2))] \in \maA_c^0 (M/F, M'/F'; \varphi)
$$
is independent of $t$, that is does not depend on the choice of $(Q,Q')$.  

\medskip

The function of $s$ whose limit defines $\Ind_{\Lam, \Lam'}(D_L^E, D_L^{E'})$,  is  constant for all  $s \geq s_1 \geq  s_0$, for large enough $s_1$, so the limit exists. Indeed, 
$$
k_{S^2}(x, x) = k_{{S'}^2} ( \varphi (x), \varphi (x)) \quad \text{and} \quad k_{R^2}(x, x) = k_{{R'}^2}(\varphi (x), \varphi (x)),
$$
outside $M(s_1)$, for large enough $s_1$, so 
$$
\Ind_{\Lam, \Lam'}(D_L^E, D_L^{E'}) \,\, = \, \int_{M \ssm M(s)}  \hspace{-1.0cm} k_{S^2} \lan x,x \ran  - k_{R^2} \lan x,x \ran \, d\mu  \,\, - \,\, 
\int_{M' \ssm M'(s)}    \hspace{-1.2cm}  k_{{S'}^2} \lan x,x \ran   - k_{{R'}^2} \lan x,x \ran \, d\mu',
$$
for any $s \geq s_1$. 

Finally, 
$$
\Ind_{\Lam, \Lam'}(D_L^E, D_L^{E'}) \,\, = \, 
\lim_{s \to \infty} \left(\int_{M \ssm M(s)}  \hspace{-1.0cm} k_{S^2} \lan x,x \ran  - k_{R^2} \lan x,x \ran \, d\mu  \,\, - \,\, 
\int_{M' \ssm M'(s)}    \hspace{-1.2cm}  k_{{S'}^2} \lan x,x \ran   - k_{{R'}^2} \lan x,x \ran \, d\mu'  \right)  \,\, = \,\,
$$
$$ 
\int_{T \ssm T_{s_1}}  \int_F  k_{S^2} \lan x,x \ran  - k_{R^2} \lan x,x \ran \, dx_F d\Lam \,\, - \,\,
\int_{T' \ssm T_{s_1}'} \int_{F'}  k_{{S'}^2} \lan x,x \ran   - k_{{R'}^2} \lan x,x \ran \, dx_{F'}d\Lam'  \,\, = \,\,
$$
$$ 
\int_{T \ssm  T_{s_1}} \hspace{-0.5cm}  \tr( k_{S^2})  - \tr(k_{R^2}) \, d\Lam\,\, - \,\, 
\int_{T' \ssm T_{s_1'}}  \hspace{-0.6cm}  \tr(k_{{S'}^2})  - \tr(k_{{R'}^2}) \, d\Lam'    \,\, = \,\, 
$$
$$ 
\lan    [(\tr(S^2) - \tr(R^2), \tr({S'}^2) - \tr({R'}^2))] , (\Lam,\Lam') \ran,
$$
which is independent of the choice of $(Q,Q')$.
\end{proof}

\medskip

Recall, Section \ref{LDOs}, that $\AS (D_L^E)$ is the Atiyah-Singer characteristic differential form for $D_L^E$, and similarly for $\AS (D_L^{E'})$.  By Theorem \ref{zerolim}, the  Schwartz kernel $k_{e^{-t(D_L^E)^2}}$ of the leafwise heat operator $e^{-t(D_L^E)^2}$ and  the leafwise characteristic form $\AS (D_L^E)_L$ satisfy
$$
\lim_{t\to 0}  \tr_s \left(k_{e^{-t(D_L^E)^2}} (x, x)\right) \,\, = \,\, \AS (D_L^E)_L(x),
$$
uniformly on $M$. 
Since 
$$
\AS (D_L ^E)_L \,|_{M \ssm \cK} = \varphi^*(\AS (D_L ^{E'})_L \, |_{M' \ssm \cK'}),
$$
the pair $(\AS (D_L^E)_L, \AS (D_L^{E'})_L$ satisfies
$$
\left[  \left( \int_F \AS (D_L ^E)_L, \int_{F'} \AS (D_L^{E'})_L \right)\right]  \,\, \in \,\,  \maA_c^0 (M/F, M'/F'; \varphi). 
$$ 

\medskip

The theorem below is the first half of the foliation relative measured index theorem.  It and Theorem \ref{mainSection4} comprise a generalization of Theorem 4.18 of \cite{GL3}, see also \cite{LM}, IV, Theorem 6.5. 
\begin{theorem}\label{Firstmain} 
The following index formula holds:  
$$
\Ind_{\Lam, \Lam'}(D_L^E, D_L^{E'}) \,\, = \,\,
\int_{\cK} \AS (D_L ^E)_L  \, d\Lam \,\, - \,\,   \int_{\cK'} \AS (D_L ^{E'})_L \, d\Lam'.
$$
\end{theorem}
Note that the right hand side equals 
$\dd \lan[( \int_F \AS (D_L ^E)_L, \int_{F'} \AS (D_L^{E'})_L)] , (\Lam, \Lam') \ran$.

\begin{proof}
We begin by constructing parametrices $Q_{t}$ and $Q_{t}'$  with $\varphi$ compatible remainders, which satisfy
$$
 \Id - Q_{t} D_L^E \,\, = \,\,  \Id - D_L^E Q_{t}  \quad \text{and} \quad \Id - Q'_{t} D_L^{E'} = \Id - D_L^{E'} Q'_{t}.
$$
It follows immediately that 
$$
k_{S_{t}^2} \lan x,x \ran   - k_{R_{t}^2} \lan x,x \ran  \,\, = \,\,  \tr_s(k_{( \Id - Q_{t} D_L^E)^2}(x,x)), 
$$
and we will show that as  $t\to 0$, $\tr_s(k_{( \Id - Q_{t} D_L^E)^2}(x,x)$
converges uniformly to $\AS (D_L^E)_L(x)$.  Of course, the same holds for $(\Id - Q'_{t} D_L^{E'})^2$.

For a real function $g$, set $g_\lambda(z) = g(\lambda z)$, for $\lambda \in \R$, denote its Fourier Transform by $\what{g}$ and $FT(g)$, its inverse transform by $\wtit{g}$ and $FT^{-1}(g)$, and  the convolution of $g$ and $h$ by $g \star h$.   We have the following facts:  
$$
FT (g_\lambda) = \frac{1}{\lambda} FT (g)_{ \frac{1}{\lambda}}; \;  FT (g \star h)= \sqrt{2\pi}FT(g) FT(h); \text{  and  } FT(\what{g}) = FT^{-1}(\what{g}) = g, \text{  if $g$ is even}.
$$

Fix  a smooth even non-negative function $\psi$ supported in $[-1, 1]$, which equals $1$ on $[-1/4, 1/4]$, which is non-increasing on $\R_+$, and whose integral over $\R$ is $1$.  Note that $FT( \what{\psi}) = \psi$ since $\psi$ is even.  The family $\frac{1}{\sqrt{t}} \what{\psi}_{\frac{1}{\sqrt{t}}}$ is an approximate identity when acting on a Schwartz function $f$ by convolution, since, up to the constant $\sqrt{2\pi}$ which we systematically ignore,
$$
\frac{1}{\sqrt{t}} \what{\psi}_{\frac{1}{\sqrt{t}}} \star f  \; = \; FT^{-1}(FT( \frac{1}{\sqrt{t}} \what{\psi}_{\frac{1}{\sqrt{t}}} \star f ))  \; = \; FT^{-1} (\psi_{\sqrt{t}}  \what{f}) \to \wtit{\what{f}} \; = \; f,
$$
in the Schwartz topology as $t \to 0$.  In fact more is true.
 \begin{lemma}\label{Schwlimit}
$
\dd \lim_{t\to 0} \left(\left[\frac{1}{\sqrt{t}} \what{\psi}_{\frac{1}{\sqrt{t}}} \star f\right]_{\sqrt{t}} \; -  \;   f_{\sqrt{t}}\right)  \; = \; 0,
$ 
in the Schwartz topology.  
\end{lemma}

\begin{proof}
We need only prove that the difference of the Fourier transforms goes to zero in the Schwartz topology. But, 
$$
FT \left(\left[\frac{1}{\sqrt{t}} \what{\psi}_{\frac{1}{\sqrt{t}}} \star f\right]_{\sqrt{t}}\right) - FT \left[f_{\sqrt{t}}\right]  =  \frac{1}{\sqrt{t}} FT(f)_{\frac{1}{\sqrt{t}}}   (\psi -1).
$$

Write  $\rho$ for $FT(f)\in \maS(\R)$.  Then, since $\psi -1$ and all its derivatives  are identically $0$ on $(-1/4, 1/4)$ and are bounded over $\R$ with the bound which can depend on the degree of the derivatives (this is 
not a problem),  there are constants $C_k$, depending on non-negative $n,m \in \Z$,  so that
\\

\begin{eqnarray*} 
\vert\vert z^n \frac{\pa^m}{\pa z^m} \left[ \frac{1}{\sqrt{t}} \rho_{\frac{1}{\sqrt{t}}}  (\psi -1)\right] \vert\vert_\infty & = & \sup_{\vert z\vert \geq 1/4} \left\vert z^n \frac{\pa^m}{\pa z^m} \left[ \frac{1}{\sqrt{t}} \rho_{\frac{1}{\sqrt{t}}} (\psi -1)\right] \right\vert \\
& \leq &  \sup_{\vert z\vert \geq 1/4}\left\vert z^n \sum_{k=0}^m C_k 
\frac{\pa^k}{\pa z^k} \left[ \frac{1}{\sqrt{t}} \rho_{\frac{1}{\sqrt{t}}} 
\right] (z) \right\vert \\
& =  &  \sup_{\vert z\vert \geq  1/4} \left\vert z^n \sum_{k=0}^m C_k 
\left(\frac{1}{\sqrt{t}}\right)^{k+1} \rho^{(k)} \left({\frac{z}{\sqrt{t}}}\right) \right\vert \\
& = & \sup_{\vert z\vert \geq \frac{1}{4\sqrt{t}}} \sqrt{t}^n \left\vert z^n \sum_{k=0}^m C_k \left(\frac{1}{\sqrt{t}}\right)^{k+1} \rho^{(k)} (z) \right\vert 
\end{eqnarray*}
For any non-negative $k\in \Z$, the function $z\mapsto z^n \rho^{(k)} (z)$ is Schwartz.   But for any Schwartz function $f$, any $N\geq 0$ and any $\eta >0$,

\noi\hspace{3.0cm}$\dd
\lim_{t\to 0^+} \frac{1}{t^N} \sup_{\vert z\vert \geq \frac{\eta}{\sqrt{t}}} \vert f (z)\vert =0,  \quad \text{thus,}  \quad
\lim_{t \to 0}\vert\vert z^n \frac{\pa^m}{\pa z^m} \left[ \frac{1}{\sqrt{t}} \rho_{\frac{1}{\sqrt{t}}} (\psi -1)\right] \vert\vert_\infty \; = \; 0.
$
\end{proof}

Set  $e(z) = e^{-z^2/2}$, and for $t > 0$, set  
$$
\chi^t(z) \; =  \; \left[\frac{1}{\sqrt{t}} \what{\psi}_{\frac{1}{\sqrt{t}}} \star e  \right]_{\sqrt{t}} (z).
$$
Then, by Lemma \ref{Schwlimit}, 
$$
\lim_{t \to 0}\left(\chi^t(z) \; - \; e^{-tz^2/2}\right) = 0, \;\; \text{ in the Schwartz topology.}
$$

Note that $\chi^t(D_L^E)$ has propagation $\leq 1$.  To see this, since $\what{e} = e$, we have that up to a constant,
$$
FT(\frac{1}{\sqrt{t}} \what{\psi}_{\frac{1}{\sqrt{t}}} \star e)  \; = \;  \psi_{\sqrt{t}} e,
$$
and by Theorem \ref{SCF&D1}, $\chi^t (D_L^E) \in \Gam_s(F,E )$.  In fact, up to a constant,
$$
\chi^t (D_L^E)   \; = \;  FT^{-1}(\psi_{\sqrt{t}} e)(\sqrt{t}D_L^E)   \; = \; \int_\R \psi (\sqrt{t} \xi)  e (\xi) \cos (\xi \sqrt{t} D_L^E) \, d\xi,
$$
since $\psi_{\sqrt{t}} e$ is even.  Setting $\eta = \sqrt{t}\xi $, the fact that $\Supp \psi\subset [-1, 1]$
gives,
$$
\chi^t (D_L^E) = \frac{1}{\sqrt{t}} \int_{\vert \eta\vert \leq 1} \psi (\eta)  e (\eta/\sqrt{t}) \cos (\eta D_L^E) d\eta.
$$
The operator $\cos (\eta D_L^E)$ has propagation $\leq \vert \eta\vert$, see \cite{Ch73, Roe87}.  Thus $\chi^t (D_L^E)$   has propagation $\leq1$.

\medskip

Set  
$$
Q_{t} \,\, = \,\,  \left(\frac{1- (\chi^t (0))^{-1} \chi^t (z)}{ z}\right) (D_L^E) \quad \text{ and }  \quad
Q'_t  \,\, = \,\,  \ \left(\frac{1- (\chi^t (0))^{-1} \chi^t (z)}{ z}\right)(D_L^{E'}).
$$
These are parametrices with $\varphi$ compatible remainders for the operators $D_L^E$ and $D_L^{E'}$. Indeed, 
$$ \Id - D_L^E Q_{t}  \,\, = \,\,  \left(1-  z\left(\frac{1- (\chi^t (0))^{-1} \chi^t (z)}{ z}\right)\right) (D_L^E)  \,\, = \,\, 
$$
$$
 \left(1-  \left(\frac{1- (\chi^t (0))^{-1} \chi^t (z)}{ z}\right)z\right) (D_L^E) \,\, = \,\, \Id - Q_{t} D_L^E  \,\, = \,\,  (\chi^t (0))^{-1} \chi^t (D_L^E)  \,\,  \in \,\, \Gam_ s (F, E).
$$
Note that,  
$
\chi^t(z)= FT(\psi_{\sqrt{t}} e)(\sqrt{t}z), 
$
as $\wtit{e} = \what{e} = e$ and $\psi_{\sqrt{t}} e$ is even.  Thus
$$
0 \,\, < \,\,  \chi^t(0) \; = \;  FT(\psi_{\sqrt{t}}\,e)(\sqrt{t} \cdot0)    \; = \;  
\frac{1}{\sqrt{2\pi}} \int_{\R}    \psi(\sqrt{t}\xi ) e^{-\xi^2/2} \, d\xi  \; \leq \;  
\frac{1}{\sqrt{2\pi}} \int_{\R}  e^{-\xi^2/2} \, d\xi  \; = \;  1.
$$
In addition, 
$$
-\lim_{z \to 0} \frac{1 - (\chi^t (0))^{-1} \chi^t (z)}{z} \,\, = \,\, (\chi^t (0))^{-1}  \lim_{z \to 0} \frac{ d\chi^t }{dz} (x) \,\, = \,\, 0,
$$
since $\dd \frac{d\chi^t}{dz}$ is odd.  For simplicity of notation, we will ignore $(\chi^t (0))^{-1}$ in what follows.  Thus 
$$
k_{S_{t}^2} \lan x,x \ran   - k_{R_{t}^2} \lan x,x \ran  \,\, = \,\,  \tr_s(k_{( \Id - Q_{t} D_L^E)^2}(x,x)) \,\, = \,\,   \tr_s(k_{\chi^t (D_L^E)^2}(x,x))  , 
$$ as claimed, and similarly for $Q'_{t}$.

Since  $\chi^t (D_L^E)$ and $\chi^t (D_L^{E'})$ have propagation $\leq 1$,   
$k_{\chi^t (D_L^E)^2 }(x,x)$  and $k_{\chi^t (D_L^{E'})^2 }(x,x)$ are  completely determined by what $(D_L^E)^2$ and $(D_L^{E'})^2$ are within a   distance $1$ of $x$.  Since $(D_L^E)^2$ and $(D_L^{E'})^2$ are  $\varphi$ related off $\cK$ and $\cK'$, the pair $\tr_s (\chi^t (D_L^E)^2) , \tr_s(\chi^t (D_L^{E'})^2)$  is $\varphi$ related off  the $1$  penumbras of $\cK$ and $\cK'$.   
Thus,
$$
\left[ \left( \tr_s (\chi^t ( D_L^E)^2 ), \tr_s(\chi^t ( D_L^{E'})^2 )\right)\right] \,\, \in \,\,
\maA_c^0 (M/F, M'/F'; \varphi), 
$$
and, we may  express the measured relative index as
$$
 \Ind_{\Lambda, \Lambda'}(D_L^E, D_L^{E'}) \,\, = \,\,    \left\lan \left[ \left( \tr_s (\chi^t ( D_L^E)^2 ), \tr_s(\chi^t ( D_L^{E'})^2 )\right)\right], (\Lam, \Lam') \right\ran. 
$$
The right hand side is independent of $t$, because of its independence of the choice of the  pair $(Q_t,Q'_t)$ by Proposition \ref{indepQ}.

Now, 
\vspace{-0.2cm}
$$
\left\lan \left[ \left( \tr_s (\chi^t ( D_L^E)^2 ), \tr_s(\chi^t ( D_L^{E'})^2 )\right)\right], (\Lam, \Lam') \right\ran \,\, = \,\,
$$
\vspace{-0.6cm}
\begin{Equation} \label{limsinf} \end{Equation}
\vspace{-0.7cm}
$$ 
\lim_{s \to \infty} \left(\int_{T \ssm T(s)}  \hspace{-0.5cm} \tr_s(\chi^t ( D_L^E)^2) \, d\Lam  \,\, - \,\, 
\int_{T' \ssm T'(s)}    \hspace{-0.7cm}  \tr_s(\chi^t ( D_L^{E'})^2) \, d\Lam'  \right) \; = \;
$$
$$ 
\int_{T \ssm T(s_1)}  \hspace{-0.5cm} \tr_s (\chi^t ( D_L^E)^2) \, d\Lam  \,\, - \,\, 
\int_{T' \ssm T'(s_1)}    \hspace{-0.7cm}  \tr_s(\chi^t ( D_L^{E'})^2) \, d\Lam',
$$
for $s_1$ sufficiently large.  The same $s_1$ works for all $t > 0$.  Note that $T \ssm T(s_1)$ and $T' \ssm T'(s_1)$ are relatively compact, so have finite volumes.

\smallskip

Let $\delta > 0$ be given.  There is $t_0 > 0$, so that for all $0 < t <  t_0$,  and all $x$,
$$
| \tr_s(k_{\chi^t(D_L^E)^2}(x,x)) -  \tr_s(k_{e^{-t(D_L^E)^2}}(x,x))| \; \leq \; \delta/2.
$$
This follows from Proposition \ref{bddedconv} and Lemma \ref{Schwlimit} and their proofs.  These give 
$$
| \lan k_{\chi^t(D_L^E)^2}(x,x) (v), v \ran  -  \lan k_{e^{-t(D_L^E)^2}}(x,x) (v), v \ran | \; \leq \; 
$$
$$
||\chi^t(D_L^E)^2 - e^{-t(D_L^E)^2})||_{\ell,-\ell}   ||\delta_x^v||_{-\ell} ||\delta_x^v||_{-\ell}  \; \leq \; 
$$
$$
\sup_{z \in \R}(|(1 + z^2)^{\ell}(\chi^t(z)^2 - e^{-tz^2}| )  ||\delta_x^v||_{-\ell}^2 \; \to 0 \text{ as  } t \to 0,
$$
independently of $\delta_x^v$.  Finally, $\tr_s(k_{\chi^t(D_L^E)^2}(x,x))$ is a finite sum of elements of the form $k_{\chi^t(D_L^E)^2}(\delta_x^v,\delta_x^v)$, as is $\tr_s(k_{e^{-t(D_L^E)^2}}(x,x))$.

From Theorem \ref{zerolim} we have, 
$$
\lim_{t \to 0} \tr_s(k_{e^{-t(D_L^E)^2}}(x,x)) \; = \; AS((D_L^E)^2)_L(x) \quad \text{uniformly}.
$$
Choose $t \in (0,t_0)$ so that 
$$
| \tr_s(k_{e^{-t(D_L^E)^2}}(x,x)) \; - \; AS((D_L^E)^2)_L(x)| \; \leq \: \delta/2.
$$ 
Then for all $x$, 
$$
| \tr_s(k_{\chi^t(D_L^E)^2}(x,x)) \; - \; AS((D_L^E)^2)_L (x) |\; \leq \: \delta.
$$
We may assume that the same holds for 
$$
|\tr_s(k_{\chi^t(D_L^{E'})^2}(x,x)) \; - \; AS((D_L^{E'})^2)_L(x).
$$
Thus
$$ 
\int_{T \ssm T(s_1)}  \hspace{-0.5cm} \tr_s \left(\chi_t (t D_L^E)^2\right) \, d\Lam  \,\, - \,\, 
\int_{T' \ssm T'(s_1)}    \hspace{-0.7cm}  \tr_s\left(\chi_t (t D_L^{E'})^2\right) \, d\Lam'
$$
differs from 
$$ 
\int_{T \ssm T(s_1)}  \hspace{-0.5cm}AS((D_L^E)^2)_L \, d\Lam  \,\, - \,\, 
\int_{T' \ssm T'(s_1)}    \hspace{-0.7cm}  AS((D_L^{E'})^2)_L \, d\Lam'
$$
by at most $\delta(\vol(T \ssm T(s_1))  + \vol(T' \ssm T'(s_1)))$, where $\delta$ is as small as we please, so they must be equal.  \end{proof} 

Denote by $\chi_{[a,b]}$ the characteristic function of the interval $[a,b]$.
For $\ep \geq 0$, denote by $P_{[0,\ep]}$ the  leafwise spectral projection $\chi_{[0,\ep]}((D_L^E)^2) $ of $(D_L^E)^2$, with leafwise Schwartz kernel $k_{P_{[0,\ep]}}$.
\begin{defi} The $\Lam$ dimension of $P_{[0,\ep]}$ is 
$$
\Dim_{\Lam}(\Im(P_{[0,\ep]}))  \,\, = \,\, \int_M k_{P_{[0,\ep]}} \lan x,x \ran\, d\mu  \,\, = \,\, 
\int_T \left[ \int_F  k_{P_{[0,\ep]}} \lan x,x \ran \, dx_F\right] d\Lam.
$$
In particular, 
$$
\Dim_{\Lam}(\Ker((D_L^E)^2)) \,\, = \,\, \Dim_{\Lam}(\Im(P_0))  \,\, = \,\, \int_M k_{P_0} \lan x,x \ran \, d\mu  \,\, = \,\, \int_T \left[ \int_F k_{P_0} \lan x,x \ran \,dx_F\right] d\Lam.
$$
\end{defi}

Since $P_{[0,\ep]}$ is a projection, it is automatically a positive operator.  It is a standard result that the function  $k_{P_{[0,\ep]}} \lan x,x \ran$ is non-negative and leafwise smooth.  

\begin{lemma}\label{MeasurableP}
 $P_{[0,\ep]}$  is transversely measureable.
\end{lemma}
\begin{proof}
First note that  $\chi_{[0,\ep]}((D_L^E)^2) = \chi_{[-\sqrt{\ep}, \sqrt{\ep}]}(D_L^E)$.
Next, recall the approximate identity $\frac{1}{\sqrt{t}} \what{\psi}_{\frac{1}{\sqrt{t}}}$ from the proof of Lemma \ref{Schwlimit}, where $\psi$ is a smooth even non-negative function  supported in $[-1, 1]$, which equals $1$ on $[-1/4, 1/4]$, which is non-increasing on $\R_+$, and whose integral over $\R$ is $1$.  Then 
$\frac{1}{\sqrt{t}} \what{\psi}_{\frac{1}{\sqrt{t}}} = FT(\psi_{\sqrt{t}})$, so $\what{\psi}(x)$ is Schwartz, and 
$$
\int (1+y^2)^\ell \frac{1}{\sqrt{t}} \what{\psi}_{\frac{1}{\sqrt{t}}}(y)\, dy
$$ 
is uniformly bounded independently of $t$. Indeed, we may assume $0<t \leq 1$, then setting $y=\sqrt{t}x$ we get
$$
\int (1+y^2)^\ell \frac{1}{\sqrt{t}} \what{\psi}_{\frac{1}{\sqrt{t}}}(y)\, dy   \leq 
\int  (1+y^2)^\ell\vert \frac{1}{\sqrt{t}} \what{\psi}_{\frac{1}{\sqrt{t}}}(y)\vert \, dy  = 
\int  (1+tx^2)^\ell \vert \what{\psi}(x)\vert\, dx \leq \int  (1+x^2)^\ell\vert\what{\psi}(x)\vert \, dx  <  +\infty,
$$
since $\what{\psi}(x)$ is Schwartz.

Let $\rho_n$ be a sequence of smooth compactly supported non-negative even functions taking values  in $[0, 1]$.  We require that $\rho_n$ be supported in $\vert y\vert\leq \sqrt{\ep}+1/n$, be  equal to $1$ on $\vert y\vert 
\leq \sqrt{\ep}$,  and converges pointwise to $\chi_{[-\sqrt{\ep}, +\sqrt{\ep}]}$ as $n \to \infty$. 
For $\ell \geq 0$,  $0 \leq (1+y^2)^\ell  \rho_n (y) \leq (1+(\sqrt{\ep}+1)^2)^\ell$.  
Peetre's inequality says that for all $y$, $z$, and $\ell$,
$$
(1+z^2)^\ell \; \leq \; 2^{|\ell|} (1+(z-y)^2)^{|\ell|}(1+y^2)^\ell. 
$$ 
Thus, for $\ell \geq 0$, the family 
$$
(1+z^2)^\ell \vert ( \frac{1}{\sqrt{t}} \what{\psi}_{\frac{1}{\sqrt{t}}} * \rho_n) (z)\vert \; = \; 
|  \int_{\R}(1+z^2)^\ell  \frac{1}{\sqrt{t}} \what{\psi}_{\frac{1}{\sqrt{t}}}(z-y)  \rho_n(y) \, dy |  \; \leq \;  
$$ 
$$
\vert  \int_{\R}  2^\ell (1+(z-y)^2)^\ell(1+y^2)^\ell \frac{1}{\sqrt{t}} \what{\psi}_{\frac{1}{\sqrt{t}}}(z-y)  \rho_n(y) \, dy  \vert  \; \leq \;  
$$
$$
\vert \int_{\R}  2^\ell (1+(z-y)^2)^\ell   \frac{1}{\sqrt{t}} \what{\psi}_{\frac{1}{\sqrt{t}}}(z-y) (1+(\sqrt{\ep}+1)^2)^\ell \, dy  \vert  \; = \;  
$$
$$
2^\ell  (1+(\sqrt{\ep}+1)^2)^\ell  \int_{\R} (1+y^2)^\ell  \frac{1}{\sqrt{t}} \what{\psi}_{\frac{1}{\sqrt{t}}}(y) \, dy, 
$$
 is also uniformly bounded independently of $t$.  
 
By Proposition \ref{bddedconv},  the Schwartz kernel of $( \frac{1}{\sqrt{t}} \what{\psi}_{\frac{1}{\sqrt{t}}} * \rho_n) (D_L^E)$ converges uniformly to the Schwartz kernel of $\rho_n (D_L^E)$ when $t \to 0$.   As the Fourier transform of $ \frac{1}{\sqrt{t}} \what{\psi}_{\frac{1}{\sqrt{t}}} * \rho_n$ is in $C_c^\infty 
(\R)$,  the Schwartz kernel of $( \frac{1}{\sqrt{t}} \what{\psi}_{\frac{1}{\sqrt{t}}} * \rho_n) (D_L^E)$ belongs to $\Gamma_s(F, E)$.   It follows from results in [HL90], Section 2, that  the Schwartz kernel of $( \frac{1}{\sqrt{t}} \what{\psi}_{\frac{1}{\sqrt{t}}} * \rho_n) (D_L^E)$, so also  the Schwartz kernel of $\rho_n (D_L^E)$, is measurable for all $n$.

Now, $\rho_n$ converges pointwise to $\chi_{[-\sqrt{\ep}, \sqrt{\ep}]}$ and using again that for $\ell \geq 0$, $(1+y^2)^\ell \vert \rho_n (y)\vert\leq (1+(\sqrt{\ep}+1)^2)^\ell$,  Proposition 3.7 gives that the Schwartz kernel of  $\rho_n (D_L^E)$ converges pointwise to the Schwartz kernel of $P_{[0,\ep]}$, so  the Schwartz kernel of $P_{[0,\ep]}$ is also measurable.
\end{proof}
Thus,  $\Dim_{\Lam}(\Im(P_{[0,\ep]}))$ is well defined as an element of $[0, \infty]$, and can potentially be $\infty$. 

\medskip

Recall the sequence $M(s)$ of open subspaces, with compact complements, of $M$ defined at the beginning of this section.  Let  $s_0$ be such that $K\subset M\smallsetminus M(s_0)$.   Recall that $\varphi$ restricts to a foliated isometry which identifies $M(s_0)$ and its foliation, with $M'(s_0)$ and its foliation. Moreover, we have the identification
$$
\phi:E\vert_{M(s_0)} \simeq E'\vert_{M'(s_0)} \text{ and also the conjugation of the Dirac operators}.
$$
The restriction $\what{D}_L^E$ of the operator $D_L^E$ to $M(s_0)\cap L$ can be defined as in \cite{GL3} by simply restricting to $dom (D_L^E)\cap L^2 (M(s_0)\cap L, E)$. However, we shall rather define $\what{D}_L^E$ so that it is a closed operator acting from the Hilbert space $L^2(M(s_0)\cap L, E)$ to itself. This is achieved by setting 
$$
dom (\what{D}_L^E):=\{\xi\in L^2(M(s_0)\cap L, E)\, \vert\, D_L^E\xi\in L^2(M(s_0)\cap L, E)\}.
$$
Since $D_L^E$, with its maximal domain (= its minimal domain) is  closed and self-adjoint, the resulting operator $\what{D}_L^E$ is  a closed symmetric operator from $L^2(M(s_0)\cap L, E)$ to itself.

The  above conjugation over $M(s_0)$ then allows the identification of $\what{D}_L^E$ with the restricted  operator $\what{D}_L^{E'}$ of $D_L^{E'}$ to $M'(s_0)$.  
Denote by $P_0(s_0)$ the orthogonal projection onto $\Ker(\what{D}_L^E$).    Similarly, we have $P'_0(s_0)$, which we can identify with  $P_0(s_0)$. 
\begin{assu}\label{Assume}  For $s > s_0$,  \;
$\dd
\int_{M(s)}  \hspace{-0.2cm}k_{P_0(s_0)} \lan x,x \ran   \,  d\mu \;  <  \; \infty.
$ 
\end{assu}
\noi This assumption is satisfied for instance when the foliation admits PSC near infinity as we shall see shortly. 

\medskip

The following is the second half of the foliation relative measured index theorem.  
\begin{theorem}\label{mainSection4}
Suppose that there is $\ep_0 > 0$ so that  $\Dim_{\Lam}(\Im(P_{[0,\ep_0]}))$ and $\Dim_{\Lam}(\Im(P'_{[0,\ep_0]}))$ are finite, and that Assumption \ref{Assume} holds. Then, for $0 < \ep \leq \ep_0$,
$$
\Ind_{\Lam, \Lam'}(D_L^E, D_L^{E'}) \,\, = \,\,
\int_M \tr_s(k_{P_{[0, \ep]}}(x,x)) \, d\mu     \,\, - \,\, \int_{M'}   \tr_s(k_{P_{[0, \ep]}'}(x,x)) \, d\mu'.
$$
Thus,
$$
\int_M \tr_s(k_{P_0}(x,x)) d\mu     \,\, - \,\, \int_{M'}   \tr_s(k_{P_0'}(x,x)) \, d\mu' \,\, = \,\,
\int_{\cK} \AS (D_L ^E)_L  \, d\Lam \,\, - \,\,   \int_{\cK'} \AS (D_L ^{E'})_L  \, d\Lam'.
$$
\end{theorem}

\begin{remark} 
Note that  $\AS (D_L^E)_L$ satisfies  $\dd\int_F \AS (D_L^E)_L \in \, \cA^0_c(T)$, while in general the integral of the global form $\dd\int_F \AS (D_L^E)  \in \, \cA^*_c(T)$, and may include higher order terms.  One might hope that Theorem \ref{mainSection4} extends to these higher order terms (and the higher order terms of the Chern characters of $P_0$ and $P_0'$), without additional restrictions on the spectral measures.  The examples in  \cite{BHW14} show that this is not the case, since they satisfy the hypothesis of Theorem \ref{mainSection4}, but not its conclusion for one of these higher order terms.  We show in \cite{BH21} that the restrictions given in \cite{HL99} and \cite{BH08} do allow for the extension to the higher order terms. 
\end{remark}

\begin{proof} 
For $\ep > 0$, denote by $\zeta_\ep:\R\to \R$ the bounded Borel function given by
$$
\zeta_\ep (x) \; = \; \frac{1}{x} \chi_{(\ep, \infty)}(x^2) \quad \text{and set} \quad G^{\ep} \; = \;  (G^{\ep}_L) \; = \; (\zeta_\ep (D_L^E)),
$$
 which is a bounded leafwise operator with  norm $\leq \frac{1}{\sqrt{\ep}}$.   Extend its leafwise Schwartz kernel $k_{G^{\ep}} (x, y)$ over $M\times M$, by defining it to be zero if $x$ and $y$ are not on the same leaf.  
\begin{lemma}\label{SmoothingGreen}
$k_{G^{\ep}}$  is  measurable and leafwise smooth off the diagonal. 
\end{lemma}

\begin{proof}
Since $(M, F)$ has bounded geometry, the operator $P_{[0, \ep]}$ is a leafwise smoothing operator, that is, it is bounded between any two leafwise Sobolev spaces, with a uniform global bound.  If $Q$ is a uniform finite propagation leafwise pseudodifferential parametrix of $D_L^E$, then  $\Id-D_L^EQ=R$ is a finite propagation leafwise smoothing bounded operator, so its kernel is in $\Gamma_s(F, E)$. Then
$$
G^{\ep} = Q-P_{[0, \ep]} Q+G^{\ep} R
$$
 where $-P_{[0, \ep]} Q+G^{\ep} R$ has a smooth uniformly bounded kernel.
Since $Q$ has a measurable leafwise smooth Schwartz kernel off the diagonal,  $k_{G^{\ep}}$ is also measurable and leafwise smooth off the diagonal.
\end{proof}

Recall the good cover $\cU = \{U_i, \psi_i\}$ of $M$, and its associated good cover $\{V_i, \psi_i\}$.
Let $\what{\psi}:M\times M \to [0,1]$ be a smooth bump function supported on $\bigcup_i V_i \times V_i$, a bounded open neighborhood of the diagonal $\Delta M$,  which is equal to $1$ on the smaller open neighborhood $\bigcup_i U_i  \times  U_i $ of $\Delta M$.  We require that all its derivatives $\pa_x^\alpha\pa_y^\beta \what{\psi} (x, y)$ in these local coordinates be uniformly bounded over $M \times M$.  Denote by $\psi:M\times M\to [0,1]$ the transversely measurable leafwise smooth function with $\psi = \what{\psi} $ on $ \bigcup_i V_i  \times_{\gam_i} V_i $,  where $\gam_i$ is the constant path at some point in $U_i$, and $\psi$ is zero otherwise.  Then $\psi$ belongs to $\Gamma_s(F, M \times \R)$ and is equal to $1$ on $ \bigcup_i U_i  \times_{\gam_i} U_i$. In particular, its leafwise derivatives are uniformly bounded over $M \times M$.  In addition, for any leaf $L$, $\psi$ restricted to  $L \times L$ is supported in an open bounded neighborhood of $\Delta L$,  and is $1$ on a smaller open neighborhood of $\Delta L$.    Denote by $Q^\ep_\psi$ the leafwise operator with leafwise Schwartz kernel 
$$
k_{Q^\ep_\psi} (x, y) \,\, = \,\,  \psi (x, y) k_{G^\ep} (x, y), 
$$
which is supported on $\bigcup_i V_i  \times_{\gam_i} V_i$.  

\begin{lemma}\label{PsiGreen}\
\begin{itemize}
\item The operators $Q^\ep_\psi$, $D_L^E Q^\ep_\psi$ and $Q^\ep_\psi D_L^E$ are uniformly bounded operators on the spaces $L^2(L, E)$.
\item  The operators $
R^\ep_\psi = \Id - D_L^E Q^\ep_\psi\text{ and } S^\ep_\psi = \Id - Q^\ep_\psi D_L^E$
have leafwise Schwartz kernels in $\Gamma_s(F, E)$. 
\end{itemize}
\end{lemma}

\begin{proof}
We may choose a leafwise pseudodifferential parametrix $Q$ for $D_L^E$ which is supported in the leafwise neighborhood of the diagonal where the function $\psi$ is identically $1$, and we may assume that it is in the uniform bounded calculus of each leaf, see \cite{KordyukovThesis}.  Then $R = \Id - D_L^E Q$ and $S = \Id - Q D_L^E$ belong to $\Gamma_s(F, E)$, and 
$$
k_{Q^\ep_\psi} (x, y) = k_Q (x, y) - \psi (x, y) k_{P_{[0, \ep]} Q} (x, y) + \psi (x, y) k_{G^\ep R} (x, y).
$$ 
By \cite{KordyukovThesis} we know that $Q$ is a uniformly bounded operator on the spaces $L^2(L, E)$.  As noted in the proof of Lemma \ref{SmoothingGreen}, $P_{[0, \ep]}$ is a leafwise smoothing operator, so also is $P_{[0, \ep]} Q$, see for instance \cite{KordyukovThesis}. Therefore,  the restriction of the Schwartz kernel $\psi (x, y) k_{P_{[0, \ep]} Q} (x, y)$ to any $L\times L$ belongs to the class of uniform finite propagation smoothing operators, that is it satisfies  [Definition 3.1] of \cite{Shubin}. In particular, it is supported in a uniformly finite distance neighborhood of the diagonal, and is uniformly bounded with all its derivatives in local normal coordinates. But such a kernel defines a leafwise smoothing operator, see again \cite{KordyukovThesis, Shubin}, so it is uniformly bounded.

Finally, by the Spectral Mapping Theorem, the operator  $G^\ep$ is  uniformly bounded between any Sobolev space and itself, so the operator $G^\ep R$ is a uniformly smoothing operator. The bounded geometry assumption then implies that $k_{G^\ep R}$ is measurable and leafwise smooth with uniform $L^\infty$-bounds on all its derivatives in local normal coordinates.  Using the properties of $\psi$,  this implies that $\psi (x, y) k_{G^\ep R} (x, y)$ is also measurable and leafwise smooth  with uniform $L^\infty$-bounds on all its derivatives in local normal coordinates. In addition, this latter kernel has uniform finite propagation and therefore belongs to $\Gamma_s(F, E)$. This proves in particular that $Q^\ep_\psi$ is uniformly bounded and that $Q^\ep_\psi - Q$ belongs to $\Gamma_s(F, E)$.

As 
$$
D_L^E Q^\ep_\psi  \, = \,\, \Id \, - \, R \, + \, D_L^E (Q^\ep_\psi - Q) \quad \text{and} \quad 
Q^\ep_\psi D_L^E  ,\, = \,\, \Id \, - \, S \, + \, (Q^\ep_\psi - Q)D_L^E,
$$
we also have that $D_L^E Q^\ep_\psi$ and $Q^\ep_\psi D_L^E$ are uniformly bounded operators on the spaces $L^2(L, E)$, and both $R^\ep_\psi$ and $S^\ep_\psi$  have leafwise Schwartz kernels in $\Gamma_s (F, E)$.
\end{proof}

Denote by $\fm_0$ multiplication by the characteristic function of $M(s_0)$, which we identify with $M'(s_0)$.  Set $\what{G}^{\ep} = \fm_0 G^{\ep} \fm_0$.  We also have the operators $ \what{G}'^\ep= \fm_0 G^{'\ep} \fm_0$,  $\fm_0\Id' \fm_0 = \fm_0\Id\fm_0$,  $\what{P}_{[0,\ep]} = \fm_0 P_{[0,\ep]} \fm_0 $, $\what{P}'_{[0,\ep]}  = \fm_0 P'_{[0,\ep]} \fm_0$,  $ \what{P}_{(\ep,\infty)} = \fm_0 (P_{(\ep,\infty)}) \fm_0$, and   $\what{P}'_{(\ep,\infty)} = \fm_0(P'_{(\ep,\infty)})\fm_0$.  Note carefully that these operators are acting on the Hilbert spaces $\cH(s_0) = (\cH_L(s_0))= (L^2(L \cap M(s_0), E |_{L \cap M(s_0)}, dx |_{L \cap M(s_))})$. 

\begin{prop}\label{AltLemma}  
For $s > s_0$, \;
$\dd
\int_{M(s)} \vert k_{\what{G}^\ep} \lan x,x \ran   -      k_{ \what{G}'^\ep} \lan x,x \ran \vert   \,  d\mu \;  <  \; \infty.
$
\end{prop}

\begin{proof}
We begin with the following.

\begin{lemma}\label{eubf}
Let $A=(A_L)$ be a $\Lambda$-essentially uniformly bounded family of self-adjoint operators, acting on the Hilbert spaces $\cH(s) = (\cH_L(s))$. Suppose $P=(P_L)$, where each $P_L$ is a self-adjoint projection on  $\cH_L$ with smooth leafwise Schwartz kernel.  Then the Schwartz kernel of $PAP$ satisfies $\vert k_{PAP}\langle x,x\rangle \vert \leq ||A|| k_P\langle x, x \rangle$, where  $||A||$ is the $\Lambda$-essential supremum over $M(s)$ of the operator norms of the operators $A_L$.
\end{lemma}

\begin{proof}
The operator $||A|| \Id_L - A_L$ is a self-adjoint non-negative operator for any $L$. Therefore, the operator $P_L (||A|| \Id_L - A_L) P_L$ is also a selfadjoint non-negative operator, and its leafwise Schwartz kernel $||A|| k_{P_L} - k_{P_LA_LP_L}$, when restricted to the diagonal is a non-negative section whose trace is a non-negative function. 
Therefore the local trace function is non-negative. But this is equal to $||A|| k_{P}\langle x, x\rangle - k_{PAP}\langle x, x \rangle$. 
Since $k_{PAP}\langle x, x\rangle$ is a real function, we can use the same argument with $-PAP$ and get $||A|| k_{P}\langle x, x\rangle \pm k_{PAP}\langle x, x \rangle \geq 0$ and hence the conclusion.
\end{proof}

Next, we adapt the proof of Lemma 4.28  in \cite{GL3},  and the material on its preceding page.

Consider  $\overline{\Im [(\what{G}^\ep -  \what{G}'^\ep)(\what P_{[0,\ep]}-\what P'_{[0,\ep]})]}$, the closure of the range of $(\what{G}^\ep -  \what{G}'^\ep)(\what P_{[0,\ep]}-\what P'_{[0,\ep]})$.   Since we are identifying $M(s_0)$ with $M'(s_0)$,  we can identify  $D_L^{E'}$ restricted to $M'(s_0)$ with $\what{D}_L^E = D_L^E$ restricted to $M(s_0)$.    Denote by  $P$ the orthogonal projection onto the  closure of the subspace 
$$
W \; = \; \Ker (\what{D}_L^E ) + \overline{\Im [(\what{G}^\ep -  \what{G}'^\ep)(\what P_{[0,\ep]}-\what P'_{[0,\ep]})]}.
$$
Then
$$
\what{G}^\ep -  \what{G}'^\ep= P (\what{G}^\ep -  \what{G}'^\ep) = P (\what{G}^\ep -  \what{G}'^\ep) P.
$$
Indeed, one has (acting on $\cH(s_0)$),
$$
\what{D}_L^E(\what{G}^{\ep}  \; - \;  \what{G}'^\ep) \; = \; 
\fm_0 (D_L^E G^{\ep}) \fm_0  \; - \; \fm_0 (D_L^{E'} G^{'\ep}) \fm_0 \; = \;
 \fm_0 (P_{(\ep,\infty)}) \fm_0  \; - \; \fm_0(P'_{(\ep,\infty)})\fm_0  \; = \;
$$ 
$$
\fm_0(\Id -P_{(\ep,\infty)})\fm_0 \; - \; \fm_0(\Id'- P'_{(\ep,\infty)})  \fm_0    \; = \;
\fm_0(P'_{[0,\ep]})\fm_0 \; - \; \fm_0(P_{[0,\ep]})  \fm_0  \; = \; \what P'_{[0,\ep]}-\what P_{[0,\ep]}. 
$$
Therefore, $(\what{G}^\ep -  \what{G}'^\ep) (\Ker (\what P_{[0,\ep]}-\what P'_{[0,\ep]})) \subset \Ker (\what{D}_L^E )$, so the range of $\what{G}^\ep -  \what{G}'^\ep$ is contained in
$$
\Im ((\what{G}^\ep -  \what{G}'^\ep)(\what P_{[0,\ep]}-\what P'_{[0,\ep]})) +  (\what{G}^\ep -  \what{G}'^\ep) (\Ker (\what P_{[0,\ep]}-\what P'_{[0,\ep]})) \subset W.
$$
Next, denote by $\rho_\ep$ the orthogonal projection onto the closure of the range of the self-adjoint operator $\what{G}^\ep -  \what{G}'^\ep$.   Then
$$
\rho_\ep P = P\rho_\ep = \rho_\ep \text{ and } \rho_\ep (\what{G}^\ep - 
 \what{G}'^\ep)=\what{G}^\ep -  \what{G}'^\ep,
$$
and so,
$$
\what{G}^\ep -  \what{G}'^\ep=  \rho_\ep (\what{G}^\ep -  \what{G}'^\ep)= P  \rho_\ep (\what{G}^\ep -  \what{G}'^\ep) = P (\what{G}^\ep -  \what{G}'^\ep).
$$
The equality $P (\what{G}^\ep -  \what{G}'^\ep) = P (\what{G}^\ep -  \what{G}'^\ep) P$  follows from the fact that all the operators are self-adjoint.

Now the norms of $\what{G}_\ep$ and $\what{G}'_\ep$ are bounded, and by Lemma \ref{eubf}, 
$$
\vert k_{\what{G}_\ep - \what{G}'_\ep} \langle x, x \rangle \vert \leq \vert\vert \what{G}_\ep - \what{G}'_\ep\vert\vert k_P\langle x, x \rangle.
$$
Thus, we only need to show that for any $s>s_0$,
$$
\int_{M(s)} k_P\langle x, x \rangle \, d\mu <+\infty.
$$
Recall that  $P_0(s_0)$ is the orthogonal projection onto  $\Ker(\what{D}_L^E)$, so  
by Assumption \ref{Assume}, this inequality follows provided we show that,
$$
\int_{M(s)} k_{P-P_0(s_0)}\langle x, x \rangle \, d\mu <+\infty,
$$
since $\Im(P_0(s_0)) \subset \Im(P)$.

To this end, consider the von Neumann algebra $W^*= W^*(M(s_0), \Lambda)$ of $F$ and $E$ restricted to $M(s_0)$ with respect to the trace $\tau_\Lambda$ associated with the restriction of $\Lambda$ to the Borel transversals in $M(s_0)$.  Denote by $r_\ep$ the leafwise orthogonal projection onto $\overline{\Im [(\what{G}^\ep -  \what{G}'^\ep)(\what P_{[0,\ep]}-\what P'_{[0,\ep]})]}$.  By the parallelogram law for projections in  $W^*$,   the orthogonal projection $P-P_0(s_0)$ is Murray-von Neumann equivalent to the orthogonal subprojection $r_\ep$.  As $\tau_\Lambda$ is constant on the Murray-von Neumann classes and non-negative, we have the estimate
$$
\int_{M(s_0)}  k_{P-P_0(s_0)} \langle x, x \rangle \, d\mu \; \leq \; \int_{M(s_0)} k_{r_\ep} \langle x, x \rangle d\, \mu.
$$
Since the image of $(\what{G}_\ep - \what{G}'_\ep)(\what P_{[0,\ep]}-\what P'_{[0,\ep]})$ is contained in 
$$
\Im \left((\what{G}_\ep - \what{G}'_\ep) \what P_{[0,\ep]}\right) + \Im \left((\what{G}_\ep - \what{G}'_\ep) \what P'_{[0,\ep]}\right),
$$
by the parallelogram law for projections, we are reduced to proving that the 
orthogonal projection onto the closure of  each of these subspaces  is $\tau_\Lambda$-trace class. 

As the proofs are the same, we only prove the first.  The subspace $\Im \left((\what{G}_\ep - \what{G}'_\ep)  \what P_{[0,\ep]}\right)$ is contained in $\Im \left((\what{G}_\ep - \what{G}'_\ep)  P_{[0,\ep]}\right)$, where the operator $(\what{G}_\ep - \what{G}'_\ep)  P_{[0,\ep]}$ is now acting leafwise in the whole manifold $M$, obtained by extending $\what{G}_\ep$ and $\what{G}'_\ep$ by zero off $M(s_0)$.   Thus,  we are reduced to checking that the orthogonal projection onto the closure of $\Im \left((\what{G}_\ep - \what{G}'_\ep)  P_{[0,\ep]}\right)$ is $\tau_\Lambda$-trace class, and has integrable leafwise Schwartz kernel. 

Recall that, by assumption,  $P_{[0,\ep]}$ (as well as $P'_{[0,\ep]}$ in $M'$) is  leafwise smoothing with finite $\tau_\Lambda$-trace. On the other hand there exists an isometry with dense range between $\Im (P_{[0,\ep]})\cap \Ker (\what{G}_\ep - \what G'_\ep)^\perp$ and the closure of the image of $(\what{G}_\ep - \what{G}'_\ep)  P_{[0,\ep]}$, acting  in the $L^2$-spaces of the leaves of $M$. A classical argument, \cite{BF}, using normality of the trace $\tau_\Lambda$ then implies that the $\tau_\Lambda$-trace of the projection onto the closure of the image of $(\what{G}_\ep - \what{G}'_\ep)  P_{[0,\ep]}$ coincides with that of the projection onto  $\Im (P_{[0,\ep]})\cap \Ker (\what{G}_\ep - \what{G}'_\ep)^\perp$. This latter is a subprojection of $P_{[0,\ep]}$, so it has finite $\tau_\Lambda$-trace and is also smoothing.   Thus,  the $\mu$-integral of its local kernel trace is finite. The conclusion follows. 

The same argument works for $(\what{G}_\ep - \what{G}'_\ep) \what P'_{[0,\ep]}$ by considering $\what{G}_\ep - \what{G}'_\ep$  in $M' (s_0)$ and extending it to $M'$.  This completes the proof of Proposition \ref{AltLemma}. 
\end{proof}

By Lemma \ref{PsiGreen} applied to $D_L^E$  and  $D_L^{E'}$, the operators $Q^\ep_\psi$ and $Q'^\ep_\psi$ are finite propagation parametrices for $(D_L^E)^+$ and $(D_L^{E'})^+$ respectively.  For simplicity, we denote these by $Q$ and $Q'$,  and by $S, R$ and $S', R'$ the corresponding remainders.

Following the proof of Theorem (1.17) in \cite{GL3},  we let $(f_n:M\to [0, 1])_{n\geq 1}$ be an increasing  sequence of measurable   compactly supported functions, which are leafwise smooth and such that: 
\begin{itemize}
\item for any compact subspace $B$ in $M$, the functions $f_n$ are identically $1$ on  $B$ for large enough $n$;
\item for all $n$ and any leaf $L$, $\vert\vert [D_L^E, f_n]\vert \vert \leq \frac{2}{n}$.
\end{itemize}
 We may assume  that each $f_n$ is equal to $1$ on a relatively compact open subspace $U$ containing $K$, such that the finite propagation operator $Q$ sends sections supported in $M\smallsetminus U$ to sections supported in $M\smallsetminus K$. Then the leafwise operator $(1-f_n\circ \varphi^{-1}) Q'$, which is well defined in $\varphi (M\smallsetminus U)$, can be transferred using the isomorphism $\Phi =(\phi, \varphi)$ to a leafwise  operator on $M\smallsetminus U$.  We denote it by $(1-f_n) Q'$. We extend it  to  an operator on $(M, F)$ by making it $0$ where it is not already defined.  We then define  new parametrices for $(D_L^E)^+$  by setting
$$
Q_n  \; = \; f_n Q \; + \; (1-f_n) Q'.
$$  The advantage is that each pair $(Q_n, Q')$ is $\varphi$-compatible.

Since the relative measured index can be computed using any pair of $\varphi$-compatible finite propagation parametrices modulo $\Gamma_s (F, E)$, we may use the remainders $R_n$, $S_n$, $R'$ and $S'$ obtained in this way, i.e.
$$
R_n=f_n R + (1-f_n) R'   \text{ and } S_n=f_n S+ (1-f_n) S'-[D_L^E, f_n] (Q-Q'),
$$
with
$$
R=\Id - Q ((D_L^E)^+), \;\; S=\Id - ((D_L^E)^+) Q,  \;\; R' =\Id - Q' ((D_L^{E'})^+), \;\;  \text{and}  \;\; S'=\Id - ((D_L^{E'})^+) Q'.
$$
By Remark \ref{nonsquare}, we can compute $\Ind_{\Lambda, \Lambda'} (D_L^E, D_L^{E'})$ by the formula below, namely using $R_n$, $S_n$, $R'$ and $S'$ in place of their squares.  Note that the formula is independent of $n$, and that,  for example, $k_{S_n}\lan x, x\ran  -  k_{R_n}\lan x, x\ran = k_{S_n - R_n}\lan x, x\ran$, etc., which simplifies the notation.

Now
$$
\Ind_{\Lambda, \Lambda'} (D_L^E, D_L^{E'})  =  
\lim_{s\to\infty}\left( \int_{M\smallsetminus M(s)}\hspace{-0.8cm} k_{S_n -R_n} \lan x,x \ran \, d\mu  -  
\int_{M'\smallsetminus M'(s)} \hspace{-0.8cm} k_{S'-R'} \lan x,x \ran  \, d\mu' \right) =
$$
$$
\lim_{s\to\infty}\left( \int_{M\smallsetminus M(s)}  \hspace{-1.0cm} f_n (x) k_{S-R}  \lan x,x \ran \, d\mu  -  \int_{M'\smallsetminus M'(s)}  \hspace{-1.0cm}  f_n(x')  k_{S'-R'}  \lan x,x \ran \,  d\mu' -
\int_{M\smallsetminus M(s)}  \hspace{-0.3cm} k_{[D_L^E, f_n] (Q-Q')}  \lan x,x \ran \, d\mu \right) =
$$
$$ \int_M f_n (x) k_{S-R}  \lan x,x \ran \, d\mu  -  \int_{M'} f_n(x') k_{S'-R'}  \lan x,x \ran \,  d\mu'  -  
\int_{M\smallsetminus K} \hspace{-0.3cm} (\nabla_{L} f_n)(x) k_{Q-Q'}  \lan x,x \ran \, d\mu.
$$
Here $f_n$ is defined on $M'$ by the transport using $\varphi$ and by defining it to be $1$ where it is not already defined.  We also used that for large enough $s$, the support of $f_n$ is contained in $M\smallsetminus M(s)$. Finally $[D_L^E, f_n]$ is the zero-th order differential operator which is Clifford multiplication by the leafwise gradient $\nabla_{L} f_n$ of $f_n$, and the compact support of $\nabla_{L} f_n$ is contained in $M\smallsetminus K$. 

Note that by our choice of parametrix defined by cutting off $G^{\ep}$ and $G^{'\ep}$, we have that on the diagonal
$$
k_{S-R}  \lan x,x \ran = \tr_s k_{P_{[0,\ep]}} (x, x)\text{ and }  k_{S'-R'}  \lan x,x \ran = \tr_s k_{P_{[0,\ep]}'} (x, x).
$$
Thus, by the Dominated Convergence Theorem, 
$$
\lim_{n\to\infty} \int_M f_n (x) k_{S-R}  \lan x,x \ran \, d\mu=\int_M \tr_s (k_{P_{[0,\ep]}} (x, x)) \, d\mu,
$$
and
$$
\lim_{n\to \infty} \int_{M'} f_n(x') k_{S'-R'}  \lan x,x \ran \, d\mu'=\int_{M'} \tr_s  k_{P_{[0,\ep]}'} (x, x)\,  d\mu'.
$$

Finally,  for large enough $n$, the support of $\nabla_L f_n$ is a subset of $M(s)$ with $s > s_0$, as large as we please.  So, we need only show that,  
$\dd \lim_{n\to\infty} \int_{M(s)} \hspace{-0.2cm} (\nabla_{L} f_n)(x) k_{Q-Q'}  \lan x,x \ran \, d\mu = 0$. 
Recall that the restriction of $k_Q$ to the diagonal coincides with that of $k_{G^\ep}$, so for large $s$, it  coincides with that of $k_{\what{G}^\ep}$ and similarly for $k_{Q'}$.  Thus, by Proposition \ref{AltLemma} ,  we have 
$$ 
\int_{M(s)}  \hspace{-0.2cm} \left\vert k_{Q-Q'}  \lan x,x \ran \right\vert  d\mu <\infty.
$$ 
As $\vert\vert \nabla_L f_n\vert\vert \leq \frac{2}{n}$  the proof is now complete for the $[0,\ep]$-projections.

The second statement with $\ep=0$ follows immediately using the Dominated Convergence Theorem and the fact that each integrand decreases as $\ep$ decreases to zero.
   
This completes the proof of Theorem \ref{mainSection4}. 
\end{proof}
 
\section{Operators with $\Lam$ finite spectral projections}\label{TIMAFTS}

We now give examples of operators which satisfy the hypotheses of Theorem \ref{mainSection4}.   

For a leafwise Dirac operator $D_L^E$,  the canonical operator $\maR^E_F$ on sections of $E_L$ is given by
$$
\maR^E_F(\varphi) \,\, = \,\, \frac{1}{2} \sum_{j,k=1}^p e_j \cdot e_k \cdot  R^E_{e_j,e_k} (\varphi),
$$ 
where $R^E$ is the curvature operator of the connection $\nabla$ on $E_L$,  $e_1,...,e_p$ is a local orthonormal basis of $TF$, and $e_j\cdot$ is Clifford multiplication.   Note that $\maR^E_F$ is well defined, leafwise smooth, and that it is globally bounded because of our assumption that $\nabla$ is of bounded geometry.   
The operators $D_L^E$  and $\maR^E_F$ are related by the general leafwise Bochner Identity, \cite{LM}, which is

\begin{Equation}\label{basic}
$
\hspace{4.9cm} (D_L^E)^2 \,\, = \,\, \nabla^* \nabla \,\, + \,\,  \mathcal{R}^{E}_F.
$ 
\end{Equation}

\medskip

This is the main result of this section.
\begin{theorem}\label{measPR}   Suppose the curvature operator $\maR_F^E$ is uniformly positive near infinity, that is, there is a compact subset $\cK \subset M$  and $\kappa_0 = \sup \{ \kappa \in \R \, | \, \maR^E_F -\kappa \Id \geq 0 \, \text{on} \, M\smallsetminus \cK\}$ is positive.  Then for $0 \leq \ep < \kappa_0$, $\Dim_{\Lam}(\Im(P_{[0,\ep]}))$  is finite.   More precisely,
$$
0 \; \leq \; \Dim_{\Lam}(\Im(P_{[0,\ep]}))   \,\, \leq\,\, \,\, \frac{(\kappa_0 - \kappa_1)}{(\kappa_0 - \ep)}  \int_{\cK}  \tr(k_{[0,\ep]}(x, x))  \,  d\mu \,\, < \,\, \infty,   
$$
where $\kappa_1  = \sup \{ \kappa \in \R \, | \, \maR^E_F -\kappa \Id \geq 0\, \text{on} \, M\}$.  In addition, Assumption \ref{Assume} holds.
\end{theorem}

Prime examples are spin foliations which admit leafwise metrics with PSC off a compact subset of $M$, with $E=\maS$ the spinor bundle associated to $TF$, or more generally, its tensor product $E=\maS\otimes E_0$ by any Hermitian bundle $E_0$ which is leafwise flat near infinity. In fact, all that is needed   is that $E_0$ defines a $K$-theory class which is leafwise almost flat near infinity.

When the $\Lambda$-dimension of $P_0$ is finite, we denote by $\Ind_\Lambda (D_L^E)$ the well defined measured index given by
$$
\Ind_\Lambda (D_L^E) \; = \; \int_M \tr_s (k_{P_0} (x, x)) d\mu,
$$
and similarly for $\Ind_{\Lambda'} (D_L^{E'})$. 

Immediate corollaries of Theorems \ref{Firstmain},  \ref{mainSection4}, and  \ref{measPR}  are the following.

\begin{theorem}\label{IMCOR1}
Suppose that  $(M, F, \cK)$ and $(M', F',\cK')$ are compatible foliations as defined at the beginning of Section \ref{RelInd}.  Suppose further that  $\maR_F^E$ is uniformly positive on $M\smallsetminus \maK$, so also  $\maR_{F'}^{E'}$ is uniformly positive on $M'\smallsetminus \maK'$.  Then 
$$
\Ind_{\Lam, \Lam'}(D_L^E, D_L^{E'}) \; = \; \Ind_\Lambda (D_L^E) - \Ind_{\Lambda'} (D_L^{E'}) \;  =  \;
\int_{\cK} \AS (D_L ^E)_L  \, d\Lam  -   \int_{\cK'} \AS (D_L ^{E'})_L  \, d\Lam'.
$$
\end{theorem}

\begin{theorem}\label{IMCOR2}
Suppose that $(M,F)$ is a foliated manifold as explained in the Introduction.  Suppose further that $E$ and $E'$ are two Clifford bundles over $M$ which are isomorphic off some compact subspace $\cK$ of $M$, and that $\cR^E_F$, so also  $\maR_{F}^{E'}$,  is uniformly positive on $M\smallsetminus \maK$.
Then 
$$
\Ind_{\Lam, \Lam}(D_L^E, D_L^{E'})  \,\, = \,\,  \Ind_\Lambda (D_L^E) - \Ind_\Lambda (D_L^{E'})
 \,\, = \,\, 
 \int_M (\AS (D_L)(\ch(E) - \ch(E'))_L  \, d\mu.
$$
\end{theorem}
Note that  $\ch(E)  = \ch(E')$ off $\cK$. 

\medskip

Returning to Theorem \ref{measPR}., note that $\kappa_1  \leq \kappa_0$, and since  $\cK$ is compact and $\maR^E_F$ is bounded, $\kappa_1$ is automatically finite.   Note also, that $P^L_0 = P_0 |_L$, so also $P^L_{[0, \ep]} = P_{[0, \ep]} |_L$, is not necessarily of trace class in the classical sense.  If a leaf $L$ passes through $\cK$ an infinite number of times, the classical trace of $P^L_0$ may be infinite.   However, if it passes through $\cK$ only a finite number of times, its classical trace must be finite by the result of Gromov and Lawson, Theorem 3.2 of \cite{GL3}. 

In the case where $\kappa_1  > 0$, we get an even stronger result.
\begin{proposition}  \label{measPR2} Suppose that $\kappa_1  > 0$.   Then for $0 \leq \ep < \kappa_1$,  $P_{[0,\ep]} = 0$.  
\end{proposition}

\begin{proof} 
Suppose that $P_{[0,\ep]} \neq 0$ for some $0 \leq \ep <  \kappa_1$.  Let  $\sigma\neq 0$ be in the image of $P_{[0,\ep]} $.
Then there is a leaf $L \subset M$ so that $\sigma_L = \sigma|_L \neq 0$ on $L$.  We may assume that the 
 $L^2$ norm $||\sigma_L ||$ of  $\sigma_L$ is $1$. 
Then, since the operator $(\ep-(D_L^E)^2) P_{[0, \ep]}^L$ is non-negative, we have
$$
\ep \,\, \geq \,\,\langle (D_L^E)^2 \sigma_L, \sigma_L\rangle\, = \, 
\int_L \langle \nabla^* \nabla \sigma_L , \sigma_L \rangle  \, dx_F \,\, + \,\, 
\int_L \langle \maR^E_F\sigma_L, \sigma_L \rangle \, dx_F  \,\, = \,\, 
$$
$$
\int_L ||\nabla\sigma_L||^2 \, dx_F\,\, + \,\, 
\int_L \langle \maR^E_F \sigma_L, \sigma_L \rangle \, dx_F
\,\, \geq  \,\, 
\int_L \langle \maR^E_F \sigma_L, \sigma_L \rangle \, dx_F   \,\, \geq  \,\, \kappa_1,
$$
an obvious contradiction.  
\end{proof}

We begin the proof of Theorem \ref{measPR} with some lemmas.
\begin{lemma} \label{finiteint}
For $0 \leq \ep <  \infty$, 
$  
\dd    0 \,\, \leq \,\,  \int_{\cK}  k_{P_{[0,\ep]}}\lan x,x \ran  \, d\mu   \,\, < \,\, \infty.
$
\end{lemma}

\begin{proof}
The first inequality is because $k_{P_{[0,\ep]}}\lan x,x \ran \geq 0$. 

Our bounded geometry assumption implies that for each leaf $L$ and $k \in \Z$, the Sobolev space $\maH^k(E_L)$  is the completion of $C^{\infty}_c(E_L)$ in the norm
$$
||\sigma_L||_k \,\, = \,\, ||(1 + (D_L^E)^2)^{k/2} \sigma_L ||_0
$$
where $|| \cdot ||_0$  is the $L^2$ norm on $C^{\infty}_c(E_L)$.  If $A: \maH^j (E_L) \to \maH^k(E_L)$ is a bounded operator, its operator norm is denoted $||A||_{j,k}$. 
The Spectral Mapping Theorem says that for any bounded Borel function $g$ on $\R$, 
$$
||g(D_L^E)||_{j,k} \,\,  = \,\,   ||(1 + (D_L^E)^2)^{(k-j)/2} g(D_L^E)||_{0,0}   \,\, \leq \,\,   \sup_{x \in \R} \, (1 + x^2)^{(k-j)/2} |g(x)|.
$$  
It is thus immediate that $||P_{[0,\ep]}^L||_{j,k}$ is finite for all $j,k$, so standard arguments in Sobolev theory give that $P_{[0,\ep]}^L$ is a smoothing operator on each leaf $L$.  So, $ k_{P_{[0,\ep]}}\lan x,x \ran $ is  leafwise smooth.

Given a vector $u$ of unit length  in the fiber of $E_L$ at $x$, define the Dirac delta section $\delta^u_x$  of $E_L$  by 
$$
\langle \delta^u_x, \sigma_L \rangle \,\, = \,\,  \langle u, \sigma_L(x) \rangle. 
$$
Bounded geometry also implies that the Sobolev norms $||\delta^u_x||_{-k}$ are bounded for $x \in M$ and  $k$ large enough.   Then we have 
$$ 
|\langle k_{P_{[0,\ep]}}(x,x)(v), u \rangle| \,\, = \,\, |\langle P_{[0,\ep]}(\delta^v_x) , \delta^u_x \rangle | \,\, \leq \,\, ||\chi_{[0,\ep]}((D_L^E)^2)||_{-k,k} ||\delta^v_x||_{-k}  ||\delta^u_x||_{-k}  
$$ 
is uniformly bounded on $M$.   Since $\cK$ is  compact, we are done. \end{proof}

The proof of Theorem \ref{measPR}  involves applying the leafwise Bochner identity to $k_{P_{[0,\ep]}}$.  When we apply an operator to the first variable, we will indicate that by the subscript $(1)$, and for the second variable by the subscript $(2)$.  For example,
$\nabla_{(1)}\nabla_{(2)} k_{P_{[0,\ep]}} \lan x,x \ran$ is shorthand for $\tr(\nabla^x\nabla^y k_{P_{[0,\ep]}} (x,y ) \, |_{y=x})$.    
\medskip

The kernel $k_{P_{[0,\ep]}}(x,y) |_{L \times L} =  \sum_{i}\sigma^L_i(x) \otimes \sigma^L_i(y)$, where $\sigma^L_1, \sigma^L_2, ...$ is an orthonormal basis of $\Im(P_{[0,\ep]} |_{L \times L} )$, and  the expression on the right is independent of the choice of the basis. 
The Spectral Mapping Theorem gives that $(\ep  -  (D_L^E)^2) P^L_{[0,\ep]}$ is a non-negative operator, so we have immediately,
\begin{lemma} \label{techlem1} For $0 \leq \ep < \infty $, \hspace{0.04cm} $\dd 0 \,\, \leq \int_M  (\ep-  (D_L^E)^2)_{(1)}k_{P_{[0,\ep]}} \lan x, x  \ran  \, d\mu.
$
\end{lemma}
 
The proof of the following is standard in the classical case, see \cite{LM}, p.\ 155. Its proof in the foliated case is a consequence of the holonomy invariance of $\Lambda$, and is given in the Appendix.
  
\begin{lemma}\label{techlem2}
For $0 \leq \ep < \infty$, \hspace{0.04cm}  $\dd  \int_M (\nabla^* \nabla)_{(1)} k_{P_{[0,\ep]}} \lan x,x \ran \, d\mu   \,\, = \,\, \int_M  \nabla_{(1)}  \nabla_{(2)}k_{P_{[0,\ep]}} \lan x,x \ran  \, d\mu. $
\end{lemma}

Note that the function $\nabla_{(1)}  \nabla_{(2)}k_{P_{[0,\ep]}} \lan x,x \ran$ is non-negative, since 
$ \sum_{i} \lan \nabla\sigma^L_i(x), \nabla\sigma^L_i(x) \ran$  converges locally uniformly on each leaf  to $\nabla_{(1)}  \nabla_{(2)}k_{P_{[0,\ep]}} \lan x,x \ran$.    See \cite{A76}, and the proof of Lemma \ref{prlem3} below.

\medskip

\noindent
{\em{Proof of Theorem \ref{measPR}.}}  Assuming that $0 \leq \ep <\kappa_0$, and using Lemmas \ref{techlem1} and \ref{techlem2}, we have
$$
0 \,\, \leq \int_M  (\ep-  (D_L^E)^2)_{(1)} k_{P_{[0,\ep]}} \lan x, x  \ran  \, d\mu \,\, = \,\,
\int_M -  \nabla_{(1)}\nabla_{(2)} k_{P_{[0,\ep]}}\lan x, x  \ran + (\ep - \mathcal{R}^E_F)_{(1)}   k_{P_{[0,\ep]}}\lan x, x  \ran \, d\mu.
$$
Suppose that the non-positive integral
$
\dd \int_M - \nabla_{(1)}  \nabla_{(2)}k_{P_{[0,\ep]}} \lan x,x \ran   \,  d\mu   =    - \infty.
$
By Lemma \ref{finiteint} and  because $\cK$ is compact, $\maR^E_F$ is bounded and $\maR^E_F \geq \kappa_1 \Id$ on $\cK$, 
\begin{Equation}\label{formula4}
$\dd  \hspace{2.0cm} 
-\infty  \,\, < \,\, \int_{\cK}   (\ep -  \mathcal{R}^E_F)_{(1)} k_{P_{[0,\ep]}} \lan x,x \ran  \, d\mu     \,\, \leq \,\
   (\ep -  \kappa_1) \int_{\cK} k_{P_{[0,\ep]}} \lan x,x \ran \, d\mu  \,\, < \,\,  \infty. 
$
\end{Equation}
\noi Since $\maR^E_F \geq  \kappa_0\Id  > 0$ on $M \ssm \cK$, the operator $\ep -  \mathcal{R}^E_F$ is non positive on $M \ssm \cK$, so  
 $(\ep -  \mathcal{R}^E_F)_{(1)} k_{P_{[0,\ep]}} \lan x,x \ran$ is also non positive on $M \ssm \cK$, and 
$$  
\int_{M \ssm \cK} (\ep- \mathcal{R}^E_F)_{(1)}k_{P_{[0,\ep]}} \lan x,x \ran  \, d\mu\,\, \leq\,\, 0.
$$
Thus
$$
0 \,\, \leq \,\, 
\int_M -  \nabla_{(1)}\nabla_{(2)} k_{P_{[0,\ep]}}\lan x, x  \ran + (\ep- \mathcal{R}^E_F)_{(1)}  k_{P_{[0,\ep]}}\lan x, x  \ran \, d\mu  \,\, = \,\,
$$
$$
\int_M - \nabla_{(1)}  \nabla_{(2)}k_{P_{[0,\ep]}} \lan x,x \ran    \, d\mu   \,\, + \,\, 
\int_{M \ssm \cK}  (\ep- \mathcal{R}^E_F)_{(1)}  k_{P_{[0,\ep]}}\lan x, x  \ran   \, d\mu  \,\, + \,\, 
\int_{\cK}   (\ep-\mathcal{R}^E_F)_{(1)}  k_{P_{[0,\ep]}}\lan x, x  \ran  \, d\mu \,\, \leq \,\,
$$
$$ 
\int_M  -\nabla_{(1)}  \nabla_{(2)}k_{P_{[0,\ep]}} \lan x,x \ran    \, d\mu   \,\, + \,\, 
\vert \int_{\cK}  (\ep - \mathcal{R}^E_F)_{(1)}k_{P_{[0,\ep]}} \lan x,x \ran  \, d\mu \vert
 \,\, = \,\, - \infty,
$$
a contradiction.  So, 
$$
-\infty \,\, < \,\, \int_M - \nabla_{(1)}  \nabla_{(2)}k_{P_{[0,\ep]}} \lan x,x \ran   \,  d\mu \,\, \leq \,\, 0.
$$

Similarly, assuming that 
$$  
\int_{M \ssm \cK} (\ep-  \mathcal{R}^E_F)_{(1)}k_{P_{[0,\ep]}} \lan x,x \ran  \, d\mu\,\, = \,\, -\infty,
$$
leads to  a contradiction.  Thus 
$
\dd -\infty  \,\, < \,\,  \int_{M \ssm \cK} (\ep-  \mathcal{R}^E_F)_{(1)}k_{P_{[0,\ep]}} \lan x,x \ran  \, d\mu \,\, \leq \,\,  0,
$
and we have
\begin{Equation}\label{formula5}
$\dd
\hspace{2.5cm}  0 \,\, \leq \,\, \int_M  \hspace{-0.2cm} (- \nabla_{(1)}  \nabla_{(2)} + (\ep-\mathcal{R}^E_F)_{(1)})k_{P_{[0,\ep]}} \lan x,x \ran \, d\mu    \,\, \leq \,\,$ \\
$$\int_{M \ssm \cK}  \hspace{-0.5cm} (\ep-\mathcal{R}^E_F)_{(1)}  k_{P_{[0,\ep]}}\lan x, x  \ran \, d\mu   \,\, + \,\,  
\int_{\cK}   (\ep-\mathcal{R}^E_F)_{(1)}  k_{P_{[0,\ep]}}\lan x, x  \ran  \, d\mu,
$$
\end{Equation}
\noindent and all of the integrals are finite.

\medskip
 
Again since  $\maR^E_F \geq \kappa_0 \Id  > 0$ on $M \ssm \cK$, we have
$$
\int_{M \ssm\cK}   (\ep -  \mathcal{R}^E_F)_{(1)} k_{P_{[0,\ep]}} \lan x,x \ran  \, d\mu   \,\, \leq  \,\, 
(\ep - \kappa_0) \int_{M \ssm\cK} k_{P_{[0,\ep]}} \lan x,x \ran \, d\mu  \,\, \leq \,\,  0, 
$$
since $\dd \int_{M \ssm\cK} k_{P_{[0,\ep]}} \lan x,x \ran \, d\mu \,\, \geq \,\, 0$.

\medskip

Combining this result with Equations \ref{formula4} and \ref{formula5},  we get for $0 \leq  \ep  < \kappa_0$,
$$
0 \,\, \leq \,\,  
\int_{M \ssm \cK} (\ep -  \mathcal{R}^E_F)_{(1)}k_{P_{[0,\ep]}}\lan x, x  \ran \,  d\mu  \,\, + \,\, \int_{\cK}  (\ep -  \mathcal{R}^E_F)_{(1)}  k_{P_{[0,\ep]}}\lan x, x  \ran \,  d\mu            \,\, \leq \,\
$$
$$
(\ep - \kappa_0 )\int_{M \ssm \cK }   k_{P_{[0,\ep]}}\lan x, x  \ran \,  d\mu  \,\, + \,\,
(\ep - \kappa_1) \int_{\cK}   k_{P_{[0,\ep]}}\lan x, x  \ran \,  d\mu    \,\, = \,\, 
$$
$$
(\ep - \kappa_0 )\int_M k_{P_{[0,\ep]}}\lan x, x  \ran \,  d\mu    \,\, + \,\,
(\kappa_0 - \kappa_1)  \int_{\cK}  k_{P_{[0,\ep]}}\lan x, x  \ran \,  d\mu.    
$$
Thus for all $0 \leq \ep < \kappa_0$,
$$
0 \,\, \leq \,\, \Dim_{\Lam}(\Im(P_{[0,\ep]}))  \,\, = \,\,   \int_M k_{P_{[0,\ep]}}\lan x, x  \ran  \, d\mu   \,\, \leq\,\, \frac{(\kappa_0 - \kappa_1)}{(\kappa_0 - \ep)}  \int_{\cK}  k_{P_{[0,\ep]}}\lan x, x  \ran  \,  d\mu    \,\, <\,\,  \infty.
$$

To finish, we show that Assumption \ref{Assume} holds, 
namely 
\begin{lemm} For $s > s_0$,  \;
$\dd
\int_{M(s)}  \hspace{-0.2cm}k_{P_0(s_0)} \lan x,x \ran   \,  d\mu \;  <  \; \infty
$.
\end{lemm}
Recall that $P_0(s_0)$ is the orthogonal projection onto  $\Ker(\what{D}_L^E)$,
where $\what{D}_L^E = D_L^E$ restricted to $M(s_0)$, as in the proof of Proposition \ref{AltLemma}.  For simplicity of notation we denote the leafwise Schwartz kernel $k_{P_0(s_0)}$ by $k$, and $\what{D}_L^E$ by $D$.

\begin{proof}
Let $s_0< s' < s$,  so $M(s')\smallsetminus M(s)$ is relatively compact.  Let  $f:M(s_0) \to [0, 1]$ be a smooth cutoff function, such that 
$$
f \, \vert_{M(s_0)\smallsetminus M(s')} \; = \; 0\; \text{ and }\; f \, \vert_{M(s)} \; = 1 \;.
$$
Since $(D^2)_{(1)} (k) \langle x, x\rangle  = 0$, we have
$$
0 \; = \; (f^2)_{(2)} \left((\nabla^*\nabla)_{(1)} + \maR_{(1)}\right) (k) \langle x, x\rangle.
$$
Since $\maR$ is a zero-th order differential operator,
$$
(f^2)_{(2)} \maR_{(1)} (k) \langle x, x \rangle = (\maR f)_{(1)} f_{(2)} (k)\langle x, x \rangle \; \geq \; \kappa_0 (f_{(1)}f_{(2)}) (k) \langle x, x \rangle.
$$
By Lemma \ref{prlem3},  there is a smooth leafwise vector field $V_{k, f}$,  so that
$$
(f^2)_{(2)} (\nabla^*\nabla)_{(1)} (k)\langle x, x \rangle = (\nabla^*\nabla)_{(1)} ((f^2)_{(2)} k) \langle x, x \rangle = \nabla_{(1)} \nabla_{(2)} ((f^2)_{(2)} (k)) \langle x, x \rangle - \div_F (V_{k, f}) (x).
$$
Moreover, $V_{k, f}$ is supported in $M(s')$, since $f_{(2)}(x,y) = f(y) = 0$ on  $M(s_0) \times (M(s_0) \ssm M(s'))$. 

Now,
\begin{eqnarray*}
\nabla_{(2)}( (f^2)_{(2)} (k) )\langle x, x \rangle & = & (\nabla f^2)_{(2)} (k) \langle x, x \rangle  + (f^2)_{(2)} \nabla_{(2)} (k) \langle x, x \rangle\\
&=& 2(f\nabla (f) \otimes \bullet)_{(2)} (k) \langle x, x\rangle + (f^2)_{(2)} \nabla_{(2)} (k) \langle x, x \rangle,
\end{eqnarray*}
where $\nabla (f) \otimes \bullet$ is Clifford multiplication by the one form $\nabla (f)$.  Applying $\nabla_{(1)}$ we get
$$
\nabla_{(1)} \nabla_{(2)} ((f^2)_{(2)} (k)) \langle x, x \rangle = 2 (f \nabla)_{(1)}(\nabla (f) \otimes \bullet)_{(2)} (k) \langle x, x\rangle + (f\nabla)_{(1)} (f\nabla)_{(2)} (k)\langle x, x \rangle.
$$
Combining these  computation shows that  on $M(s')$,

$$
0  \; \geq \;  \kappa_0 (f_{(1)}f_{(2)}) (k) \langle x, x \rangle  \; + \; 2 (f \nabla)_{(1)}(\nabla (f) \otimes \bullet)_{(2)} (k) \langle x, x\rangle  \; + \; 
$$
$$ (f\nabla)_{(1)} (f\nabla)_{(2)} (k)\langle x, x \rangle \;  - \; \div_F (V_{k, f})(x).
$$
This equation is in fact valid over all of $M$ since the RHS extends smoothly by zero off $M(s')$.

Both $k\langle x, x \rangle$  and $\nabla_{(1)}\nabla_{(2)} (k)\langle x, x \rangle$   are non-negative functions on 
$M(s_0)$ which coincide with the non-negative functions $(f_{(1)}f_{(2)}) (k) \langle x, x \rangle$ and $(f\nabla)_{(1)} (f\nabla)_{(2)} (k)\langle x, x \rangle$ on $M(s)$.  Thus we have
$$
\kappa_0 \int_{M(s)} k\langle x, x \rangle \, d\mu \; \leq  \; \kappa_0 \int_{M(s)} k\langle x, x \rangle\,  d\mu \; + \;  \int_{M(s)} \nabla_{(1)}\nabla_{(2)} (k)\langle x, x \rangle d\mu  \;  \leq 
$$
$$
\; \kappa_0 \int_{M} (f_{(1)}f_{(2)})  k\langle x, x \rangle\,  d\mu \; + \;  
\int_{M} (f\nabla)_{(1)} (f\nabla)_{(2)}  (k)\langle x, x \rangle d\mu  \;  \leq 
$$
$$
 \int_M \div_F (V_{k, f}) (x) \, d\mu + 2 \Big\vert \int_{M(s')\smallsetminus M(s)}  \hspace{-1.0cm}  (f\nabla)_{(1)} (\nabla (f)\otimes \bullet)_{(2)} (k) \langle x, x \rangle \Big\vert \, d\mu,
$$
since $\nabla (f)$ is supported in $M(s')\smallsetminus M(s)$.  
But, 
$$
\int_M \div_F (V_{k, f}) (x) \, d\mu = 0,
$$
see proof of Lemma  \ref{prlem3}.
Since $\vert f\vert \leq 1$ and $\nabla(f)$ is uniformly bounded (say by $C_{s'}$), being zero off  the relatively compact subspace $M(s')\smallsetminus M(s)$, we have
$$
0 \; \leq \; \kappa_0 \int_{M(s)} k\langle x, x \rangle \, d\mu \; \leq \; 2\Big\vert \int_{M(s')\smallsetminus M(s)}  \hspace{-1.0cm} (f\nabla)_{(1)} (\nabla (f)\otimes \bullet)_{(2)} (k) \langle x, x \rangle  \Big\vert \, d\mu \; \leq \;   C_{s'} \int_{M(s')\smallsetminus M(s)}  \hspace{-1.0cm} 2 \vert \nabla_{(1)} (k) \langle x, x\rangle \vert \, d\mu. 
$$

Write $k\lan x,x \ran =  \sum_{i} \lan \sigma^L_i(x), \sigma^L_i(x) \ran$, where $\sigma^L_1, \sigma^L_2, ...$ is a leafwise orthonormal basis of $\Im(P_0(s_0))$.  Then using the inequality
$
2\vert \lan \nabla(\sigma^L_i)(x),\sigma^L_i(x)\vert  \ran\; \leq \; \|\nabla(\sigma^L_i)(x)\|^2 \; + \; 
\|(\sigma^L_i(x)\|^2,
$
and summing over $i$, we get
$$
2\vert \nabla_{(1)} (k) \langle x, x\rangle\vert \; \leq \; \nabla_{(1)} \nabla_{(2)} (k) \langle x, x \rangle + k \langle x, x \rangle.
$$
Thus, 
$$
0 \; \leq \; \int_{M(s)} k\langle x, x \rangle \,  d\mu \; \leq \; \frac{C_{s'}}{ \kappa_0}\int_{M(s')\smallsetminus M(s)}  \hspace{-1.0cm} \nabla_{(1)}\nabla_{(2)} (k)\langle x, x \rangle + k\langle x, x \rangle  \, d\mu \: < \; \infty,
$$
since $\mu$ is a Borel measure (so finite on compacts) and $\nabla_{(1)}\nabla_{(2)} (k)\langle x, x \rangle \, + \,  k\langle x, x \rangle$ is  non-negative and bounded on $M(s')\smallsetminus M(s)$.
 \end{proof}
 This completes the proof of Theorem \ref{measPR}.
 \hfill $\Box $

\medskip

We devote the rest of this section to two corollaries of Theorem  \ref{measPR}.  In particular, we relate
our definition of the relative index to the cut-and-paste definition considered in Section 4 of \cite{GL3}.  We consider compatible foliations  $(M, F)$ and $(M', F')$ as defined at the beginning of Section \ref{RelInd}.   

For the first corollary, we say that the foliation $F$, and so also $F'$, is reflective if there is a compact hypersurface which is transverse to $F$, and which separates off the infinite part of $V$.   For simplicity we will assume that this submanifold is just $\pa \cK$, and similarly for $\pa \cK'$.  Then we can ``cut and paste" as in \cite{GL3}.  In particular, there is $\delta \in \R$, and a neighborhood  of $\pa \cK$ which is diffeomorphic to $\pa\cK \times [-\delta, \delta]$, and so that $F$ restricted to $\pa\cK \times [-\delta, \delta]$ has leaves of the form $(L \cap \pa\cK) \times[-\delta, \delta]$.  We may assume that the foliation preserving diffeomorphism $\varphi$ extends to $\pa\cK \times  [-\delta,\delta]$, and  that  $\varphi(\pa\cK \times [-\delta, \delta])$ is diffeomorphic to $\pa\cK' \times [-\delta, \delta]$, and that it has the same properties as $\pa\cK \times [-\delta, \delta]$.  Then we have the compact foliated manifold 
$$
\what{M}  \,\,= \,\, \cK \cup_{\what{\varphi}} \cK',
$$
where $\what{\varphi}:\pa\cK \times [-\delta , \delta] \to \pa\cK' \times [-\delta, \delta]$ is given   by   $\what{\varphi} (x,s) = \varphi(x,-s)$.  We change the orientation of $F'$ to the opposite of what it was originally.  The resulting foliation $F \cup_{\what{\varphi}} F'$ is denoted $\what{F}$.  Denote by $\pi: \pa\cK \times [-\delta, \delta] \to \pa\cK$ the projection and   note that  $E \, |_{ \pa\cK \times [-\delta, \delta]} \simeq \pi^*(E \, |_{\pa\cK})$, and  $TF \, |_{ \pa\cK \times [-\delta, \delta]} \simeq \pi^*(TF \, |_{\pa\cK})$.  (Note that    $\dim(TF \, |_{\pa\cK}) = \dim (TF)$, not $\dim (TF) -1 = \dim F \, |_{\pa\cK}$.)   We may assume that  $\nabla$ and $D_L^E$ are preserved under the maps $(x,s) \to (x,-s)$ and $E_{(x,s)} \to E_{(x,-s)}$. This implies that $D_L^E$ and $D_L^{E'}$ are identified under the  glueing map used in defining $\what{M}$ and the objects on it.  In addition, $\Lam$ and $\Lam'$ fit together, giving $\what{\Lam}$. 
Finally, denote the leafwise operator on $\what{F}$ by $\what{D}$, and the  projection onto the kernel of  $\what{D}^2$ by $\what{P}_0$.  Given this situation, Alain Connes defined the measured index, 
$
\dd \Ind_{\what{\Lam}}(\what{D}) \; = \; \int_{\what{M}} \tr_s(k_{\what{P}_0}(x,x)) \, d\what{\mu},
$
which satisfies his celebrated index theorem, see  \cite{ConnesIntegration},  relating $\Ind_{\what{\Lam}}(\what{D})$ to the pairing of the usual characteristic classes with the Ruelle-Sullivan current.   That is
$$
\Ind_{\what{\Lam}}(\what{D}) \; = \; \int_{\what{M}}  \AS (\what{D}_L) \, d\what{\Lam}.
$$
We have the following immediately. 
\begin{theorem}\label{ExRelInd2} Suppose that $F$ is reflective and $\cR^E_F$  is strictly positive off $\cK$, so also $F'$ is reflective and $\maR_{F'}^{E'}$ is  strictly positive off $\cK'$.  Then  
$$
\Ind_{\Lam, \Lam'}(D_L^E, D_L^{E'})   \,\, = \,\,  \Ind_{\what{\Lam}}(\what{D})   \,\, = \,\,
\int_{\cK} \AS (D_L ^E)_L  \, d\Lam  -   \int_{\cK'} \AS (D_L ^{E'})_L  \, d\Lam'.
$$
\end{theorem}

Note that, since $(\AS (D_L^E)_L)  |_V  =  \varphi^*((\AS (D_L^{E'})_L)  |_{V'})$, this result is independent of the choice of the transverse compact hypersurface. 

\medskip 

The previous construction  extends to the following more general situation to yield the so called measured $\Phi$ relative index theorem, see again \cite{GL3}.  In particular, we assume that $(M,F)$ and $(M',F')$ satisfy the hypotheses of Theorem \ref{Firstmain}, with the following changes. In particular,  $M \ssm \cK = V_+ \cup V_{\Phi}$  and  $M' \ssm \cK'  = V_+' \cup V_{\Phi}'$, where the unions are disjoint.  For this case,  $\Phi = (\phi, \sigma)$ is a  bundle morphism from $E\to V_{\Phi}$ to $E'\to V_{\Phi}'$ as in Section \ref{RelInd}, our good covers $\cU$ and $\cU'$ are compatible on $V_{\Phi}$ and $V_{\Phi}'$, and $\Lam$ and $\Lam'$ are $\varphi$ compatible on $\cU_{V_{\Phi}}$ and $\cU'_{V'_{\Phi}}$.  Finally, we  assume that $F$ is reflective on $V_{\Phi}$, so $F'$ is reflective on $V_{\Phi}'$, and that  $\cR_F^E$ and $ \cR_{F'}^{E'}$ are strictly positive off $\cK$ and $\cK'$.  

Next, consider the manifold 
$\what{M} = (M \ssm V_{\Phi}) \cup_{\varphi} (M' \ssm V'_{\Phi})$, with the foliation
$$
\what{F} \; = \;  (F |_{M \ssm V_{\Phi}}) \cup_{\what{\varphi}} (F' |_{M' \ssm V'_{\Phi}}), 
$$ 
where the orientation on $\what{F}  |_{M \ssm V_{\Phi}} $ is the one on $F$, and that on $\what{F} |_{M' \ssm V'_{\Phi}} $ is the opposite of the  one on $F'$.  We also have the bundle $\what{E} \to \what{M}$ induced by $E$ and $E'$, the leafwise operator  $\what{D}_L^{\what{E}}$ induced by $D_L^E$  and $D_L^{E'}$,  and the invariant  transverse measure $\what{\Lam}$ induced by $\Lam$ and $\Lam'$.

Because of the positivity off compact subsets, all three operators have finite  invariant  transverse measure indices thanks to Theorem \ref{measPR}.  We then have the $\Phi$ relative index theorem.
\begin{theorem}\label{RelIndThm2}  Under the conditions above,
$$
\Ind_{\what{\Lam}}(\what{D}_L^{\what{E}}) \, = \,  
\Ind_{\Lam}(D_L^E) \,\, - \,\, \Ind_{\Lam'}(D_L^{E'}).
$$
\end{theorem} 
The proof  follows easily from Theorem \ref{measPR} by adapting the proof of Theorem 4.35 of  \cite{GL3}.

\section{Spin foliations, PSC, and spaces of PSC metrics}\label{modspaces}

We now show how to extend the Gromov-Lawson construction in \cite{GL3}, Section 3, see also \cite{LM}, IV.7,  to get an invariant for the  space of  PSC metrics on a foliation whose tangent bundle $TF$ admits a spin structure.    We calculate this invariant for a large collection of spin foliations, and show that the space of  PSC metrics on each of these foliations has infinitely many path connected components.

We still assume that $(M,F)$ admits an invariant transverse Borel measure $\Lam$, and for simplicity, we assume that $M$ is compact.  Denote by $\cR$ the space of all smooth metrics on  $F$ with the $C^\infty$ topology, and  by $\cR_{sc}^+  \subset \cR$ the subspace of PSC metrics.  In this section, contrary to previous sections, we assume that the dimension of $F$ is odd.  

Scalar curvature and the so called Atiyah-Singer operator are intimately related.  Denote by $\maS$ the canonical spin bundle associated to the spin structure on $TF$, with connection $\nabla$.    The leafwise Atiyah-Singer operator, namely the leafwise spin Dirac operator $D_L^{\cS}$,  
 $$
D_L:L^2(\maS_L )  \to L^2(\maS_L )    \quad \text{is given by} \quad  
D_L (\sigma)  \,\, = \,\,  \sum_{j=1}^p e_j \cdot \nabla_{e_j} \sigma, 
$$
where $e_1,...,e_p$ is an orthonormal local framing of $TF$.   Denote by $\kappa$ the leafwise scalar curvature of $F$, that is 
$$
\kappa = -\sum_{i,j =1}^p \langle R_{e_i,e_j}(e_i),e_j \rangle,
$$
where $R$ is the curvature operator associated to the metric on the leaves of $F$.  
In this case the Bochner Identity is quite simple, see \cite{LM}, namely
\begin{Equation}\label{basicLisch}
$
\hspace{4.9cm} D_L^2 \,\, = \,\, \nabla^* \nabla \,\, + \,\,  \frac{1}{4}\kappa.
$ 
\end{Equation}
Consider the even dimensional foliation $F_{\R}$ on $M \times \R$ with leaves $L_{\R} = L \times \R$.   If $\maU$ is a good cover of $M$,  $\cU_{\R} = \{(U \times (n-1,n+1), T) | \, (U,T) \in \maU, n \in \Z \}$ is a good cover of $M \times \R$.  $\Lam_{\R} = \Lam$ is an   invariant transverse Borel measure for $F_{\R}$, with associated global measure $d\mu_{\R}  = d\mu \times dt$.  Suppose that $g_0, g_1 \in \cR_{sc}^+$,   and $(g_t)_{t \in [0,1]}$ is a smooth family in  $\cR$ from   $g_0$ to $g_1$.  On $F_{\R}$, set  $G = g_0 + dt^2$ for $t \leq 0$, $G = g_1 + dt^2 $ for $t \geq 1$, and $G = g_t + dt^2$ for $0 < t < 1$.  

The leafwise spin Dirac operator $D_L$ extends to the leafwise spin Dirac operator $D_{\R}$ on $F_{\R}$.   Denote projection onto the kernel of  $D_{\R}^2$ by $P_0$. 
We define $i_{\Lam}(g_0,g_1) \in \R$ by 
$$
i_{\Lam}(g_0,g_1) \,\, = \,\, \int_{M \times \R} \tr_s(k_{P_0}(x,x)) \, d\mu_{\R}.
$$
This is well defined thanks to Theorem \ref{measPR}, since the metric on $F_{\R}$ has PSC off the compact subset $M \times [0,1]$.

\begin{theorem} \label{Phithm} $i_{\Lam}(g_0,g_1)$  depends only on $g_0$ and $g_1$.
If $i_{\Lam}(g_0,g_1) \,\, \neq \,\, 0$, then $g_0$ and $g_1$ are not in the same path connected component of $\cR_{sc}^+$.  
\end{theorem}

\begin{proof}
 Suppose that $g_t$ and $\what{g}_t$ are two smooth families of metrics in $ \cR$ from $g_0$ to $g_1$, with associated metrics $G$ and $\what{G}$  and associated operators $D_{\R}$ and $\what{D}_{\R}$.  Since, $G$ and $\what{G}$ have uniformly PSC off $M \times [0,1]$,  $F_{\R}$ is $\Phi$ related to itself there. Theorem \ref{IMCOR1} gives
$$
i_{\Lam}(g_0,g_1)(G)  \,\, - \,\, i_{\Lam}(g_0,g_1)(\what{G}) \,\, = \,\, 
\int_{M \times [0,1]}   (\what{A}(TF_{\R})_{G}  -  \what{A}(TF_{\R})_{\what{G}})_{L_{\R}} \,  d \Lam_{\R},
$$
where $\what{A}(TF_{\R})_{G} = AS(D_{\R})$ is the Atiyah-Singer characteristic differential form, the so-called $A$-hat form, on $M \times \R$ associated to the metric $G$, and similarly  for  $\what{G}$.  The forms $\what{A}(TF_{\R})_{G}$ and  $\what{A}(TF_{\R})_{\what{G}}$ are locally computable in terms of their associated curvatures.  Thus, off $M \times [0,1]$, they agree, which justifies the last equality.  In addition, their difference is an exact form $d\Psi$ which is locally computable in terms of their curvatures and connections.   In particular, $\Psi = 0$ on the closure of open sets where their connections agree.   So off $M \times (0,1)$, $\Psi$  is zero, since the connections agree off $M \times [0,1]$.    Applying Stokes' Theorem, we get $i_{\Lam}(g_0,g_1)(G)  \,\, - \,\, i_{\Lam}(g_0,g_1)(\what{G}) = 0$.

For the second part, assume that $g_0$ and $g_1$ are in the same path connected component of $\cR_{sc}^+$, and that $g_t$, is a  smooth family of metrics in $\cR_{sc}^+$ from $g_0$ to $g_1$.   Then $G$ restricted to each leaf of $F_{\R}$ has PSC, so  Proposition \ref{measPR2} gives that  $P_0 = 0$, and $i_{\Lam}(g_0,g_1) = 0$ also.   
\end{proof}

\begin{remark}
Theorem \ref{Phithm}  remains true if we consider concordance classes of metrics, which a priori is stronger.  Recall that leafwise metrics are concordant if there is a metric $G$ on $TF_{\R}$ so that it agrees with $g_0$ near $-\infty$ and with $g_1$ near $+\infty$.  The conclusion is that if  $i_{\Lam}(g_0,g_1) \,\, \neq \,\, 0$, then $g_0$ and $g_1$ are not in the same concordance class of metrics in $\cR^+_{sc}$.  The proof being essentially the same. 
\end{remark}

\begin{remark}
We could also extend this theory to concordance classes of leafwise flat connections $\nabla$ on an auxiliary bundle $E$.  The invariant would become $i_{\Lam}((g_0, \nabla_0),(g_1, \nabla_1))$. See \cite{Be20}.  The theorem would then be that if $g_0$ and $g_1$ are concordant, and $\nabla_0$ and $\nabla_1$ can be joined by leafwise flat connections, then 
$i_{\Lam}((g_0, \nabla_0),(g_1, \nabla_1)) = 0$.
\end{remark}

Next, we have a corollary of Theorem \ref{RelIndThm2}.
\begin{coro}\label{cor6.4} Suppose $g_0, g_1, g_2 \in  \cR_{sc}^+$.
Then
$$
i_{\Lam}(g_0,g_1) + i_{\Lam}(g_1,g_0) \,\, =  \,\, 0,
\;\, \text{and} \;\; 
i_{\Lam}(g_0,g_1) + i_{\Lam}(g_1,g_2) \,\, =  \,\,  i_{\Lam}(g_0,g_2),  \; \text{so,}
$$
$$
i_{\Lam}(g_0,g_1) + i_{\Lam}(g_1,g_2) +i_{\Lam}(g_2,g_0) \,\, =  \,\, 0.
$$
\end{coro}

\begin{proof}  In the notation of Theorem \ref{RelIndThm2}, take $(M,F)$,  $(M',F')$ and $(\what{M}, \what{F})$ to be $(M \times \R, F_{\R})$, $\cK = \cK' = M \times [0,1]$, $V_{\Phi} =  V_{\Phi}' =  M \times (-\infty, 0)$, and $V_+ = V_+' = M \times (1,\infty)$.
To compute $i_{\Lam}(g_i,g_j)$ take 
$$
G_{i,j} = g_i + dt^2 \text{  for $t \in (-\infty,0]$, and  } G_{i,j} = g_j + dt^2  \text{  for  }t \in [1,\infty).
$$ 

The first equation is obvious since $i_{\Lam}(g_0,g_1)  = \Ind_{\Lam_{\R}}(D_{\R}(G_{0,1}))$, while $i_{\Lam}(g_1,g_0)$ is  the same, except with the orientation of $\R$ reversed, which changes the sign of the resulting index.  

For the second, we have
$$
i_{\Lam}(g_0,g_1) - i_{\Lam}(g_0,g_2)  =   \Ind_{\Lam_{\R}}(D_{\R}(G_{0,1})) -   \Ind_{\Lam_{\R}}(D_{\R}(G_{0,2}))  =    \Ind_{\Lam_{\R}}(D_{\R}(G_{2,1}))= i_{\Lam}(g_2,g_1)  =  - i_{\Lam}(g_1,g_2).
$$
The second equality is from Theorem \ref{RelIndThm2}, where $\what{D}_L^{\what{E}} = D_{\R}(G_{2,1})$, $D_L^E = D_{\R}(G_{0,1})$, and $D_L^{E'} = D_{\R}(G_{0,2})$.

The third equation is now also obvious.
\end{proof} 

Now suppose that $M$ is the boundary of a compact manifold $W$ with a spin foliation $\what{F}$  which is transverse to $M$, and  which restricts to $F$ there.   Suppose further that $\Lam$ extends to an invariant  transverse measure $\what{\Lam}$ on $W$.   Extend the foliation $\what{F}$ and the metric $\what{\Lam}$ as above to  $W \cup_{M}  (M \times [0,\infty))$.  Given a metric $g$  of PSC on $F$, extend it to a complete leafwise metric $\what{g}$ on $\what{F}$  by making it $g + dt^2$ on $M \times [0,\infty)$ and extending it arbitrarily over the interior of $W$.   Then the kernel of the leafwise Atiyah-Singer operator $D_L^2$ on $\what{F}$, which is an even operator, has finite $\what{\Lam}$ dimension.  Denote the super projection onto this kernel by $P_0 = P_0^+ \oplus P_0^-$.   Then $D_L$  has a finite $\what{\Lam}$ index, namely 
$$
\Ind_{\what{\Lam}}(D_L) \,\, = \,\, \Dim_{\what{\Lam}}(\Im P_0^+) \,\, - \,\,  \Dim_{\what{\Lam}}(\Im P_0^-).
$$  
\begin{defi}\label{defig}
\hspace{0.5cm}$i_{\Lam}(g,W) \,\, = \,\, \Ind_{\what{\Lam}}(D_L)$.
\end{defi}
Note that Theorem \ref{mainSection4} and the proof of Theorem \ref{Phithm}, adapted,  show that  $i_{\Lam}(g,W)$ does not depend on the extension of $g$ over $W$, and that if $g$ extends with PSC, then $i_{\Lam}(g,W) = 0$.  
We have the following corollary of  Theorem \ref{mainSection4}.  
\begin{coro}\label{cor6.6}  Suppose that $g_0, g_1 \in \cR_{sc}^+$.  Then
$$
i_{\Lam}(g_0,W)  \,\, + \,\, i_{\Lam}(g_0,g_1) \,\, =  \,\, i_{\Lam}(g_1,W).
$$
\end{coro}

\begin{proof}
Consider the foliated manifolds  
$$
\what{M}  = M \times \R  \quad \text{ and } \quad \what{M}' = M_0 \, \dot\cup \, M_1,
$$
 which satisfy the following.
\begin{itemize}
\item $\what{M}$ has the metric $g_t$ above, giving $i_{\Lam}(g_0,g_1)$.
\item $M_0 = W_0 \cup_M  (M \times [0,\infty))$ with the metric $g_0 + dt^2$ on $M \times [0,\infty)$, and the metric $\what{g}$ on $W_0 = W$.  Take the opposite orientation on $M_0$ by reversing the orientations on $[0,\infty)$ and $W_0$, so this gives $-i_{\Lam}(g_0,W)$.
\item $M_1 = W_1 \cup_M  (M \times [0,\infty))$ with the metric $g_t$ above on $M \times [0,\infty)$, and the metric $\what{g}$ on $W_1 = W$.  Note that  $W_1\cup_M  (M \times [0,1]) = W \cup_M  (M \times [0,1]) \simeq W$, and  the metric on $M \times [1,\infty)$ is $g_1 \times dt^2$, giving $i_{\Lam}(g_1,W)$. 
\item Note that $\what{M}$ has the compact subset $K = M \times [0,1]$, that  $\what{M}'$ has the compact subset $K' = K_0  \cup K_1$, where $K_0 = W_0$ and $K_1 =  W_1 \cup_M  (M \times [0,1])$, and that there is PSC off these compact subsets.  
\end{itemize}
On $\what{M}$ and $\what{M}'$ respectively, we have the operators denoted $\what{D}$ and $\what{D}'$, and the invariant transverse measures $\what{\Lam}$ and $\what{\Lam}'$.    Note that $\what{M} \ssm K$ in $\Phi$ equivalent to $\what{M}' \ssm K' $,  $W_0$ is $\Phi$ equivalent to   $W_1$, except that the orientations are opposites, and $K_1 \ssm W_1$ is $\Phi$ equivalent to $K$.  Then, Theorem \ref{mainSection4} gives
$$
\Ind_{\what{\Lam}',\what{\Lam}}(\what{D}',\what{D})  \; = \;  \int_{K'}  AS(\what{D}')_L  \,d \what{\Lam}'\; - \; \int_{K} AS(\what{D})_L \, d \what{\Lam} \; = \; 
$$
$$
 \int_{K_1} AS(\what{D}')_L  \, d \what{\Lam}' \; - \;   \int_{K_0} AS(\what{D}')_L  \, d \what{\Lam}' \; - \; \int_{K} AS(\what{D})_L \, d \what{\Lam} \; = \; 
$$
$$
\int_{K_1\ssm W_1}  \hspace{-0.6cm}    AS(\what{D}')_L  \, d \what{\Lam}' \; - \; \int_{K} AS(\what{D})_L \, d \what{\Lam} \; = \;  0.
$$
But,  $\Ind_{\what{\Lam}', \what{\Lam}} (\what{D}',\what{D}) = \Ind_{\what{\Lam}'} (\what{D}') -\Ind_{\what{\Lam}} (\what{D})$, and 
$\Ind_{\what{\Lam}'} (\what{D}') =  i_{\Lam}(g_1,W) -  i_{\Lam}(g_0,W)$, while $\Ind_{\what{\Lam}} (\what{D}))= i_{\Lam}(g_0,g_1)$. 
\end{proof}

\medskip

\noi
{\bf Examples.}
Suppose that  $M$ is a compact $4\ell$ dimensional spin manifold with $\what{A}(M) \neq 0$, and that its fundamental group $\Gam$ acts smoothly on a compact oriented Riemannian manifold, preserving its volume form.  Such manifolds abound.  (The following  is thanks to Stephan Stolz.)  In particular, take any finitely presented group $\Gam$ which acts smoothly preserving the volume form on a compact oriented Riemannian manifold $N$.  Any finitely presented subgroup of $SO_n$ will do. Use a presentation of $\Gam$ to produce a finite 2-dimensional CW complex with fundamental group $\Gam$.  Embed the complex  in $\R^{4\ell+1}$ and thicken it into a compact manifold of dimension $4\ell+1$ with boundary $M_1$. Then $\pi_1(M_1) =  \Gam$. It is a framed manifold, so its $\what{A}$-genus is zero. Let $M_2$ be any simply connected spin manifold of dimension $4\ell$ with non-zero $\what{A}$-genus, e.g., take  $M_2$ to be the product of $\ell$ copies of the Kummer surface,  which has $\what{A}$-genus $2$, so $M_2$ has $\what{A}$-genus $2^{\ell}$.  More generally,  Lemma 5.1 of  \cite{GHS18} and its proof show that for any positive integer $\ell$ and any integer $k$, $k$ even if $\ell$  is odd, there is a closed simply connected spin manifold $M_2$ of dimension $4\ell$ with $\what{A}$-genus $k$.  Then the connected sum  $M = M_1 \, \#  \, M_2$ has the required properties.  

\medskip

Denote the universal cover of $M$ by $\wtit{M}$, and consider the flat fiber bundle 
$$
 Y =\wtit{M} \times_{\Gam} N,
$$
with its natural flat foliation $F$.   Since $\Gam$ preserves the volume form on $N$,  it descends to an invariant  transverse measure $\Lam$ for $F$.  It follows immediately that 
$$
 \int_{Y}  \what{A}(TF)_L \, d \Lam \,\, = \,\,  \what{A}(M) \vol(N)   \,\, \neq \,\, 0.
 $$
 
Recall the sequence of $SO_4$ vector bundles $\pi:E_k \to \S^4$  from \cite{GL3}, given after their Corollary 4.45.  
Using the standard metric on the base $\S^4$, an orthogonal connection on $E_k$, which gives a splitting $TE_k = T\R^4 \oplus \pi^*(T\S^4)$, and an $SO_4$ invariant metric on the fibers $\R^4$, they construct a ``torpedo" metric on the total space of $E_k$ as follows.   The metric on  $\pi^*(T\S^4)$ is the pull-back from the base.  The fibers $\R^4$ are totally geodesic and the metric on them is a smoothing near the equator, which is the same on all radial lines, of the $\S^4$ hemispherical metric on $\D^4$, attached along the equator $\pa \D^4 = \S^3$, to the cylindrical metric on $\S^3 \times (0,\infty)$.   
 
Denote by $X_k$ and $\Sigma_k = \pa X_k$, the unit disk and unit sphere bundles of $E_k$.   Note  that on $\Sigma_k \times (1 - \ep,\infty)$  the metric is  $g_k \times dt^2$, where $g_k$ on $\Sigma_k$ has PSC, which we can make as large as we please by multiplying $g_k$ by a small enough constant.   Note also that $g_k$ extends over $X_k$ with PSC.

The bundle $E_k$ is chosen so that it  has Euler number $1$, which implies that $\Sigma_k$ is homotopy sphere, and that the Pontrjagin number $p_1(E_k)$ satisfies $p_1(E_k)^2 =   4 + 896k$.  There are an infinite number of such integers $k$ so that $4 + 896k$ is a perfect square.  In particular, for $m \in \Z$, set $k=m+56 m^2$.

Denote by   
$\what{X}_k = X_k \cup_{\Sigma_k} \D^8$ the compact manifold obtained by attaching an $8$ disk along the boundary $\Sigma_k$.  Classical results of Milnor, \cite{M56, M65}, imply  that  the signature of $\what{X}_k = 1$, that $p_1(E_k)^2 = p_1( \what{X}_k)^2$,  that $\Sigma_k$ is diffeomorphic to the standard $\S^7$, (the Milnor invariant $\mu(\Sigma_k,s) = 0$),  and that  $\what{A}(\what{X}_k) = k$.  

Now consider $Y \times \S^7$ with the foliation $\what{F}  = F \oplus T\S^7$.     
Multiply the metric  $g_k$ on $\S^7$ by a constant so small  that the metric $G_k$ on $\what{F}$ has PSC everywhere.  Thus we have a countable family of PSC metrics on $\what{F}$.  Now, $Y \times \S^7 = Y \times \pa X_k$, and we have the foliation $\what{F}_k  = F \oplus TX_k$ of $Y \times X_k$, along with the transverse measure induced by $\Lam$.  
Since $G_k$ extends over $\what{F}_k$ with PSC,  the invariant $ i_{\Lam}(G_k,Y \times X_k)$ is zero.    

 \begin{theo} For $k_1 \neq k_2$, the metrics $G_{k_1}$ and $G_{k_2}$ are not in the same path connected component of the metrics of PSC on the foliation $\what{F}$ of  $Y \times \S^7$.  Thus the space $\cR_{sc}^+$ of  PSC metrics for $\what{F}$ has infinitely many path connected components.
\end{theo}

\begin{proof}
Note that $G_0$ is associated to the canonical constant curvature metric $g_0$ on $\S^7$.    Since $ i_{\Lam}(G_k, Y \times X_k) = 0$,  we have by Corollaries  \ref{cor6.4} and \ref{cor6.6},   
$$
i_{\Lam}(G_0, Y \times X_k) \,\, = \,\,  i_{\Lam}(G_k, Y \times X_k)   \,\, +  \,\,  i_{\Lam}(G_k,G_0)   \,\, = \,\,   i_{\Lam}(G_k,G_0).
$$
Theorem \ref{mainSection4} applied to $Y \times (X_k \cup_{\S^7} (\S^7 \times [1,\infty))$ and 
$Y \times (X_0 \cup_{\S^7} (\S^7 \times [1,\infty))$, using the metric $G_0$ on $Y \times \S^7$, shows that 
$$
i_{\Lam}(G_0, Y \times X_k) \,\, = \,\, i_{\Lam}(G_0, Y \times X_k) \,\, - \,\,  i_{\Lam}(G_0, Y \times X_0) \,\, = \,\,
$$
$$
\int_{Y \times X_k} \hspace{-0.5cm} \what{A}(T\what{F}_k)_L  \, d\Lam   \,\, = \,\,
 \what{A}(M) \vol(N)\what{A}(\what{X}_k) \,\, = \,\, \what{A}(M) \vol(N)k.
 $$
 For the proof that $\dd \int_{X_k} \what{A}(TX_k)  = \what{A}(\what{X}_k)$, see \cite{GL3},  the proof of Theorem 4.47.
Thus, for $k_1 \neq k_2$,
$$
i_{\Lam}(G_{k_1},G_{k_2}) \,\, = \,\, i_{\Lam}(G_{k_1},G_0)  \,\, - \,\, i_{\Lam}(G_{k_2}, G_0) \,\, = \,\,
$$
$ 
\hspace{3.0cm}  i_{\Lam}(G_0, Y \times X_{k_1}) \,\, - \,\,  i_{\Lam}(G_0, Y\times X_{k_2}) \,\, = \,\, 
 \what{A}(M) \vol(N)(k_1 - k_2) \,\, \neq \,\, 0.
$
\end{proof}

\section{Appendix}

We first justify the claim in Remark \ref{nonsquare}.

The following is standard, see \cite{T81, Roebook88}.  Note that they work on compact manifolds, but the extensions to the bounded geometry case are straightforward.

\begin{lemma}\label{Mollifier}\
Assume that the  manifold $M$, the foliation $F$, and the bundle $E$ have bounded geometry. Then there exists a family $(J_\ep)_{\ep \in (0,1)}$ of leafwise operators on sections of $E$, which are leafwise smoothing, with kernels in $\Gamma_s (F, E)$ such 
that:
\begin{enumerate}
\item The families $(J_\ep)_{\ep \in (0,1)}$ and $([J_\ep, D_L^E])_{\ep \in (0,1)}$ 
are uniformly bounded operators on any leafwise Sobolev space of $E$ 
with an $\ep$-independent bound;
\item As $\ep\to 0$, $J_\ep \to \Id$ weakly on every $L^2(L, E |_L)$.
\end{enumerate}
\end{lemma}

\begin{proposition}\label{MolProp}
For any $s\in \R$, the commutator $[J_\ep, (D_L^E)^\ell]$  is  uniformly bounded as an operator from any leafwise $s$ Sobolev space to the leafwise $s-\ell+1$ Sobolev space, with a bound which is independent of $\ep$. 

For any leafwise smoothing  operator $A$ with kernel in $\Gamma_s(F, E)$, and any $\ell \in \N$, the family $(AJ_\ep (D_L^E)^{\ell})_{\ep \in (0,1)}$ has uniformly bounded leafwise $L^2$-operator norm, with the bound being uniform  in $\ep$. 
\end{proposition}

\begin{proof}
Note that
$$
[J_\ep, (D_L^E)^\ell] = \sum_{j=0}^{\ell-1} (D_L^E)^j [J_\ep, D_L^E] (D_L^E)^{\ell-j-1},
$$
which proves the first statement.  For the second, we have
$$
AJ_\ep (D_L^E)^{\ell} = (A(D_L^E)^{\ell}) J_\ep + A [J_\ep, (D_L^E)^\ell],
$$
which is a leafwise smoothing operator, since $A$ and $A(D_L^E)^{\ell}$ are. That the leafwise $L^2$-operator norm of $AJ_\ep (D_L^E)^{\ell}$ is uniformly bounded independently of $\ep$ is now clear.
\end{proof}

\begin{corollary}\label{Mollifying}
For any smoothing operator $A$ with kernel in $\Gamma_s(F, E)$ and  any finite propagation leafwise $0$-th order operator $B$, the Schwartz kernel of $AJ_\ep B$  converges uniformly to the Schwartz kernel of $AB$ when $\ep\to 0$.   The same holds for the Schwartz kernels of $J_\ep BA$ and  $J_\ep A$.
\end{corollary}

This  immediately implies that 
$$
\lim_{\ep \to 0}  \htr(k_{AJ_\ep B} (x,x)) \; = \;  \htr(k_{AB} (x,x)),  \; \text{ and }  \; \lim_{\ep \to 0}  \htr(k_{J_\ep BA} (x,x)) \; = \;  \htr(k_{BA} (x,x)), 
$$
in $\cA_c^0(M/F)$, the space of Haefliger functions.

\begin{proof}
As in the proof of Proposition \ref{ptwseconv}, we have the following for any $\ell \geq 0$.
\begin{eqnarray*}
\langle k_{AJ_\ep B} (x, y)(w), v\rangle & = & \langle AJ_\ep B (\delta_y^w),  \delta_x^v \rangle \\
&=& \langle (\Id + (D_L^E)^2)^{\ell} AJ_\ep (\Id + (D_L^E)^2)^\ell (\Id +(D_L^E)^2)^{-\ell} B (\delta_y^w), (\Id + (D_L^E)^2)^{-\ell} (\delta_x^v) \rangle 
\end{eqnarray*}
The operator $ (\Id + (D_L^E)^2)^{\ell} A$ satisfies the assumptions for $A$ in Proposition \ref{MolProp}, so the family  
$$
(\Id + (D_L^E)^2)^{\ell} AJ_\ep (\Id + (D_L^E)^2)^\ell
$$
has uniform bounded leafwise $L^2$-operator norm which is uniform in $\ep$. This implies in turn by a $3\ep$ argument with the Schwartz inequality that $(\Id + (D_L^E)^2)^{\ell} AJ_\ep (\Id + (D_L^E)^2)^\ell$ converges weakly to $(\Id + (D_L^E)^2)^{\ell} A (\Id + (D_L^E)^2)^\ell$. On the other hand, bounded geometry implies that the delta sections live in some Sobolev space, so also  does the section $B (\delta_y^w)$, since $B$ is bounded on every Sobolev space. Hence,  there exists $\ell \geq 0$ such that $(\Id +(D_L^E)^2)^{-\ell} B (\delta_y^w)$ and $(\Id + (D_L^E)^2)^{-\ell}  (\delta_x^v)$ both belong to the Hilbert space of leafwise $L^2$ sections.  In addition, their $L^2$ norms are globally bounded. Therefore, $\langle k_{AJ_\ep B} (x, y)(w), v\rangle$ converges to $\langle k_{AB} (x, y)(w), v\rangle$ and this convergence is uniform over $M$.  The same sort argument works for $J_\ep A$  and  $J_\ep BA$.
\end{proof}

We can now justify Remark 4.4.

\medskip

By Definition 4.2 and the remark right after that definition, the $s$-limit only depends on the Haefliger functions, as far as the pair of functions is a compatible pair.   Now, the pairs composed of the integrals over the leaves of $(S-S^2, S'-S'^2)$ and $(R-R^2, R'-R'^2)$ respectively, are compatibles pairs of Haefliger functions. Thus we only need to prove that the integral over the leaves of the traces of the Schwartz kernels of $S-S^2$ and $R-R^2$ agree in $\maA_c (M/F)$, since that will also hold for $S'-S'^2$ and $R'-R'^2$ in $\maA_c (M'/F')$.

To simplify the notation, we will write $D$ for $D_L^E$.   Then,  
$$
S-S^2 \; = \; (\Id -S)S \; = \; QDS,   \quad\;   R- R^2 = (\Id -R)R \; = \; DQR, \quad \text{and} \quad   SQ  \; = \;  QR,
$$
and we have, 
$$
k_{J_\ep QDS} \to k_{QDS} \;\text{  and  }  \;   k_{DSJ_\ep Q} \to k_{DSQ}   \text{ uniformly as $\ep \to 0$}.
$$
This follows from Corollary \ref{Mollifying} by setting  $A =  DS$ and $B = Q$, (recall that $Q$ has finite propagation). 
Therefore, as  Haefliger functions, we get 
$$
\htr (k_{S - S^2}) \; = \;  \htr (k_{QDS}) \; = \;  \lim_{\ep \to 0} \htr (k_{J_\ep QDS})  \; = \;  \lim_{\ep \to 0} \htr (k_{DSJ_\ep Q})\; = \; 
$$
$$
 \htr (k_{DSQ}) \; = \; \htr (k_{DQR}) \; = \; \htr (k_{R - R^2}),
$$
in $\cA_c^0(M/F)$.  The third equality follows from Theorem 3.1 as  $J_\ep Q$ and  $DS$ are  in $\Gamma_s(F, E)$.

\medskip

Finally, we prove Lemma \ref{techlem2}, and for that we need the following. 
\begin{lemma}\label{prlem3}
Define the measurable section $V_{\ep}$ of $TF$ by setting
$$
\lan V_{\ep}, W \ran (x)  \,\, = \,\, (\nabla_W)_{(1)}k_{P_{[0,\ep]}}  \lan x, x \ran,
$$
for  any smooth  section $W$ of $TF$. 
Then the following pointwise relation between  measurable leafwise smooth functions on $M$ holds,
$$
 (\nabla^*\nabla)_{(1)} k_{P_{[0,\ep]}}  \lan x, x \ran \,\, = \,\, \nabla_{(1)} \nabla_{(2)} k_{P_{[0,\ep]}} \lan x, x \ran - \div_F (V_{\ep})(x),
$$
where for any leafwise vector field $V$, $\div_F (V)$ is its leafwise divergence.
\end{lemma}

\begin{proof}
Consider  $k_{P_{[0,\ep]}}(x,y) |_{L \times L} =  \sum_i\sigma^L_i(x) \otimes \sigma^L_i(y)$, where  $ \sigma^L_1$, $ \sigma^L_2$ \ldots is a leafwise orthonormal basis of $\Im(P_{[0,\ep]})$.
By standard arguments, \cite{GL3}, proof of Theorem 4.18, the series $\sum_i \sigma^L_i \otimes \sigma^L_i$ converges locally $C^\infty$-uniformly to $k_{P_{[0,\ep]}}$ over $L \times L$.  Therefore, the series $\sum_i \nabla^*\nabla \sigma^L_i \otimes \sigma^L_i$ (resp.\ $\sum_i \nabla \sigma^L_i \otimes \nabla \sigma^L_i$) also converges locally uniformly to the kernel  $(\nabla^*\nabla)_{(1)} k_{P_{[0,\ep]}}$ (resp.\ $\nabla_{(1)} \nabla_{(2)} k_{P_{[0,\ep]}}$) over $L \times L$.  Both limits are independent of the choice of the orthonormal basis $\sigma^L_i$. As a consequence, the  series of smooth functions on $L$, $x\mapsto \sum_i \lan \nabla^*\nabla \sigma^L_i (x) , \sigma^L_i (x)\ran$ (resp.\ $x\mapsto \sum_i \lan \nabla \sigma^L_i , \nabla \sigma^L_i(x)\ran$)  converges locally uniformly on $L$ to the smooth function $x\mapsto \nabla^*\nabla_{(1)} k^L_{[0,\ep]} \lan x,x       \ran$ (resp.\ $x\mapsto \nabla_{(1)} \nabla_{(2)} k^L_{[0,\ep]} \lan x,y \ran$).  Using a classical local computation on $L$, see \cite{LM}, p.\ 155, we have that as smooth functions on $L$, 
\begin{Equation}\label{divergence}
$\hspace{3.0cm}  \lan \nabla^*\nabla \sigma^L_i (x) , \sigma^L_i (x) \ran \,\, = \,\, \lan \nabla \sigma^L_i  (x) , \nabla \sigma^L_i  (x)\ran - \div_F (V_{\lan \sigma^L_i, \sigma^L_i \ran}) (x),$
\end{Equation}
\noindent  where $V_{\lan \sigma^L_i, \sigma^L_i \ran}$ is the section of $TF$ satisfying 
$
\lan V_{\lan \sigma^L_i, \sigma^L_i \ran} , W \ran (x)  \,\, = \,\, (\nabla_W)_{(1)} \lan \sigma^L_i  (x), \sigma^L_i  (x)\rangle,
$
for  any smooth  section $W$ of $TF$. 

For any leafwise tangent vector field $W$, the series 
$\sum_i \langle \nabla_W\sigma^L_i , \nabla_W\sigma^L_i \rangle$ converges locally uniformly to the smooth function on $L$ given by $x\mapsto (\nabla_W)_{(2)}(\nabla_W)_{(1)} k_{P_{[0,\ep]}} \lan x, x \ran$. Therefore, summing  Equation \ref{divergence} over $i$, we get, for all $x \in M$,
$$
(\nabla^*\nabla)_{(1)} k_{P_{[0,\ep]}}  \lan x, x \ran \,\, = \,\, \nabla_{(1)} \nabla_{(2)} k_{P_{[0,\ep]}} \lan x, x \ran - \div_F (V_{\ep})(x).$$
Note that all the terms in this equality are transversally Borel and leafwise smooth. 
\end{proof}

\noindent {\em Proof of  Lemma \ref{techlem2}. }   To prove that for  $0 \leq \ep < \infty$, 
$$
\int_M (\nabla^* \nabla)_{(1)} k_{P_{[0,\ep]}} \lan x,x \ran \, d\mu   \,\, = \,\, \int_M  \nabla_{(1)}  \nabla_{(2)}k_{P_{[0,\ep]}} \lan x,x \ran  \, d\mu,
$$ 
we need only observe that for any leafwise vector field $V$, the top degree leafwise form $\div_F (V) dx_F$ coincides with the leafwise exact differential form $d_F (i_V dx_F)$ with $i_V$ contraction with $V$, and that it is a standard result in foliation theory that $\dd \int_F  d_F (i_V dx_F) = 0$ in $H_c^0(M/F)$.  Thus we have  
$$
\int_M (\nabla^*\nabla)_{(1)} k_{P_{[0,\ep]}}  \lan x, x \ran  \, d\mu \,\, = \,\, \int_M\nabla_{(1)} \nabla_{(2)} k_{P_{[0,\ep]}} \lan x, x \ran - \div_F (V_{\ep})(x)  \, d\mu   \,\, = \,\,
$$
$$
\int_M\nabla_{(1)} \nabla_{(2)} k_{P_{[0,\ep]}} \lan x, x \ran   \, d\mu  \,\, - \,\,  \int_T \Big( \int_F  d_F (i_V dx_F) \Big) \, d\Lam \,\, = \,\,
\int_M\nabla_{(1)} \nabla_{(2)} k_{P_{[0,\ep]}} \lan x, x \ran   \, d\mu,
$$
and so  Lemma \ref{techlem2}.  \hfill $\Box$

\end{document}